\documentclass[review]{elsarticle}
\usepackage{exscale,fancyhdr}
\usepackage{fixmath} 
\usepackage{latexsym,epsf,a4wide}
\usepackage{amsmath}
\usepackage{mathtools}
\usepackage{amssymb}
\usepackage{graphicx}
\usepackage{colordvi}
\usepackage{graphics,pstricks}
\usepackage{graphpap,epsfig,pifont,epic,eepic}
\usepackage{float}
\usepackage{array}

\newtheorem{theorem}{Theorem}[section]
\newtheorem{definition}[theorem]{Definition}
\newtheorem{assumption}{Assumption}

\newproof{proof}{Proof}

\newdefinition{remark}{Remark}

\newcommand{\eps}{\varepsilon}
\renewcommand*{\phi}{\varphi}
\newcommand{\IR}{\mathbb{R}}

\newcommand{\IN}{\mathbb{N}}

\newcommand{\neuezeile}{\hfill~}

\renewcommand{\emph}[1]{\textsl{#1}}

\allowdisplaybreaks

\title{Rigorous Homogenization of a Stokes-Nernst-Planck-Poisson Problem for various Boundary Conditions}

\author[rvt]{N.~Ray\corref{cor1}}
\author[focal]{A.~Muntean}
\author[rvt]{P.~Knabner}
\cortext[cor1]{Corresponding author, Email: ray@am.uni-erlangen.de}

\address[rvt]{Department of Mathematics, Chair of Applied Mathematics I, Friedrich-Alexander University of Erlangen-Nuremberg, Martensstra\ss e 3, 91058 Erlangen, Germany, Email: ray@am.uni-erlangen.de, knabner@am.uni-erlangen.de }
\address[focal]{Center for Analysis, Scientific computing and Applications (CASA),
Institute for Complex Molecular Systems (ICMS), Department of
Mathematics and Computer Science, Technical University Eindhoven,
Eindhoven, The Netherlands, Email: a.muntean@tue.nl}

\begin{document}

\begin{abstract}
We perform the periodic homogenization (i.\,e.~$\eps\to 0$) of the
non-stationary Stokes-Nernst-Planck-Poisson system using two-scale
convergence, where $\eps$ is a suitable scale parameter. The
objective is to investigate the influence of \textsl{different boundary conditions and variable choices
of scalings in~$\eps$} of the microscopic system of partial
differential equations  on the structure of the (upscaled)
limit model equations. Due to the specific nonlinear coupling of the
underlying equations, special attention has to be paid when passing
to the limit in the electrostatic drift term. As a direct result of the
homogenization procedure, various classes of upscaled model
equations are obtained.
\end{abstract}


\frontmatter
\maketitle

\section{Introduction}\label{SEC:Introduction}

This paper deals with with the periodic homogenization of a
non-stationary Stokes-Nernst-Planck-Poisson-type system~(SNPP). The real-world
applications that fit to this context include areas of colloid
chemistry, electro-hydrodynamics and semiconductor devices. Our interest lies in the
theoretical understanding of colloid enhanced contaminant transport in the soil. Colloidal particles are under consideration for quite a long time since they are very
important in multiple applications ranging from waste water
treatment, food industry, to printing, etc. The monograph of van de
Ven~\cite{VandeVen} and the books by Elimelech~\cite{Elimelech} and Hunter~\cite{Hunter} yield a well
founded description of colloidal particles and their properties.
However, the different processes determining the dynamics of
colloids within a heterogenous porous medium are not yet completely
understood. Therefore, the mathematically founded forecast of contaminant transport within soils is
still very difficult, as it is strongly influenced by the movement and distribution of colloidal particles (cf.
e.\,g.~\cite{Totsche04}).

Using mathematical homogenization theory, different kinds of coupled
models have been investi\-gated/de\-rived. Besides the combination of
fluid flow and convective-diffusive transport, the coupling among
different kinds of species by chemical reactions have been discussed
for example in~\cite{Hornung}, see also the references cited
therein. Further cross couplings of the water flow by heat, chemical
or electrostatical transport are studied formally in~\cite{Auriault94}. It is worth pointing out a totally different
context, where a nonlinear coupling quite  analogous to the one of
our problem occurs -- the phase-field models of Allen-Cahn type, see~\cite{Eck04} for more details on the modeling, analysis, and
averaging of such models. Investigations concerning variable scaling
and their influence on the limit equations is illustrated (by means
of formal two-scale asymptotic homogenization) in~\cite{Auriault95}, where different choices of ranges of the
P\'{e}clet number are considered. In the same spirit, but this time
rigorously, different scale ranges are examined for a linear
diffusion-reaction system with interfacial exchange in~\cite{Peter08}. Moreover, hybrid mixture theory has been applied
to swelling porous media with charged particles in~\cite{Bennethum02I} and~\cite{Bennethum02II}. Formal
upscaling attempts of the Nernst-Planck-Poisson system using formal
asymptotic expansion are reported, for instance, in~\cite{Auriault94},~\cite{Looker},~\cite{Moyne02} and~\cite{Moyne06}. It is worth pointing out that~\cite{Moyne02}
and~\cite{Moyne06} succeed to compute (again formally)
microstructure effects on the deforming, swelling clay. In spite of
such a good formal asymptotic understanding of the situation,
rigorous homogenization results seem to be lacking. Only recently, Schmuck published a paper concerning the rigorous upscaling
of a non-scaled Stokes-Nernst-Planck-Poisson system with transmission conditions for the electrostatic potential,~\cite{Schmuck11}. Furthermore, Allaire et al. studied the
stationary and linearized case in~\cite{Allaire10}. 
Our paper
contributes in this direction since we perform the rigorous
homogenization of the SNPP system
for different boundary conditions as well as for variable choices of scalings in~$\eps$, where~$\eps$ is a scale
parameter referring to a (periodically-distributed) microstructure.
The main focus of the paper thereby lies on the investigation of the
influence of the boundary condition and scalings in~$\eps$ on the structure of the effective
limit equations. This paper is built on~\cite{Ray}. However, we corrected essential errors concerning the use of Poicar\'{e}'s inequality. Furthermore, we introduce suitable redefinitions of the electrostatic potential in order to provide a more clearly arranged form of our homogenization results. Most important for the applications, we extend our results for different choices of boundary conditions for the electrostatic potential and include Stokes equations to our analysis in order to describe the interactions with the fluid flow.

The paper is organized in the following way: In
Section~\ref{SEC:MathMod}, we present the underlying microscopic
model equations -- the Stokes-Nernst-Planck-Poisson system. This is the
starting point of our investigations. The Nernst-Planck equations
describe the transport (diffusion, convection and electrostatic drift) of
and reaction between (number) densities of colloidal particles. The
electrostatic potential is given as a solution of Poisson's equation
with the charge density which is created by the colloidal particles
as forcing term. The fluid flow is determined by a modified Stokes equation. Basic results concerning existence and
uniqueness of weak solutions of this coupled system of partial
differential equations are stated in Theorem~\ref{THM:ExistenceUniqueness} in
Section~\ref{SEC:PoreScaleModel}. Moreover,
Section~\ref{SEC:PoreScaleModel} contains the definition of the
basic heterogenous and periodic geometric setting. The (small) scale
parameter~$\eps$ introduced here balances different physical terms
in the system of partial differential equations and plays a crucial
role in the homogenization procedure. Furthermore,~$\eps$ independent \emph{a priori} estimates are shown for both Neumann and Dirichlet boundary conditions of the electrostatic potential in Theorem~\ref{THM:UniformAPrioriEstimatesNeumann} and Theorem~\ref{THM:UniformAPrioriEstimatesDirichlet}. In
Section~\ref{SEC:TwoScaleConv}, we state the basic definitions and
well known compactness results concerning the method of two-scale
convergence. The main idea is to obtain an ``equivalent'' system of
partial differential equations that can reasonably describe the
effective macroscopic behavior of the considered phenomena.  We
achieve this by investigating rigorously the limit~$\eps\rightarrow
0$ using two-scale convergence. Our analysis focuses on the
influence of the choice of the boundary condition for the electrostatic potential and the different choices of scalings in~$\eps$ on both the
\emph{a priori} estimates and the structure of the limit problems.
The main calculations are included in Section~\ref{SEC:UpscalingNeumann}
and Section~\ref{SEC:UpscalingDirichlet}. The crucial point is the nonlinear
coupling of the system of partial differential equations by means of
the electrostatic potential, and therefore, the passage to the limit
$\eps\rightarrow 0$ in the nonlinear transport terms of the
Nernst-Planck equations and the Stokes equation. The main result (Theorems~\ref{THM:limitsPhiNeumann},~\ref{THM:limits V Neumann},~\ref{THM:limits c Neumann} and Theorems~\ref{THM:LimitsPhiDirichlet},~\ref{THM:limitsVDirichlet},~\ref{THM:limits c Dirichlet}) of the paper discuss for which choices of scaling we can pass
rigorously to the limit $\eps\rightarrow 0$. The results of this
homogenization procedure and the structure of the limit equations are emphasized in
Remarks~\ref{REM:limitsPhiN},~\ref{REM:limitsVN},~\ref{REM:limitsCN} and~\ref{REM:limitsPhiD},~\ref{REM:limitsVD},~\ref{REM:limitsCD} and in Section~\ref{SEC:Discussion}.

\section{The Underlying Physical Model}\label{SEC:MathMod}

We list in Table~\ref{TAB:Variablen} all variables and physical
parameters that are used in the following including
their dimensions. Thereby,~$L$ is a unit of length,~$T$ a unit of
time,~$M$ stands for a unit of mass,~$C$ for a unit of charge, while~$K$ represents the unit of temperature.
\begin{table}
\centering\fbox{ $
\begin{array}{llr}
v & [L/T] & \text{velocity}\\
p & [M/L/T^2]& \text{pressure}\\
\eta &  [M/L/T]& \text{kinematic viscosity of the fluid}\\
\rho &  [M/L^3]& \text{density of the fluid}\\
c &  [1/L^3]& \text{number density}\\
D &  [L^2/T]& \text{diffusivity}\\
\nu &  [-]& \text{outer unit normal}\\
\Phi &  [V]:=[ML^2/T^2/C]& \text{electrostatic potential}\\
\sigma & [ML/T^2/C] & \text{surface charge density}\\
z &  [-]& \text{charge number}\\
e &  [C]& \text{elementary charge}\\
\epsilon_0\epsilon_r & [C/V/L]& \text{dielectrostatic permittivity $\cdot$ relative permittivity}\\
k &  [ML^2/T^2/K]& \text{Boltzmann constant}\\
T &  [K]& \text{absolute temperature}
\end{array}
$}
\caption{List of the variables and physical parameters and their dimensions.}\label{TAB:Variablen}
\end{table}

In this section, we formulate a system of partial differential
equations describing colloid dynamics. Following e.\,g.~\cite{Elimelech} and~\cite{VandeVen}, we impose to our
system the balance of mass as well as the conservation of electrostatical
charges. Note that in most applications, colloidal particles are
charged~\cite{VandeVen}. Besides standard transport mechanisms
(convection and diffusion), a charged dispersion of colloidal
particles is also transported by the electrostatic field created by the
particles themselves as well as by the possibly charged soil matrix. Further interaction potentials (e.\,g.
van-der-Waals forces or an externally applied electrostatic
field) may also act on the colloidal particles. Throughout this
paper we neglect the latter effects and focus on the investigations
of the intrinsic electrostatic interaction. Following Chapter~3.3 in~\cite{VandeVen}, the positively~(+) and negatively~(-) charged particles are modeled in an Eulerian approach by some number density~$c^\pm$, which is transported by the total velocity~$v^\pm$ that
consists of two parts: First, the convective velocity term~$v^{\text{hydr}}$ due to the fluid flow within the porous medium in
which the colloidal particles are transported. This is the same for
all types of charge carriers. Second, the drift term~$v^{\text{drift},^\pm}$, that is different for both kinds of charge
carriers, can be calculated from the drift force~$F^{\text{drift},\pm}
= -z^\pm e\nabla\Phi$ via
\begin{equation*}
v^{\text{drift},\pm} = f^\pm F^{\text{drift},\pm} = -f^\pm z^\pm e\nabla\Phi
\end{equation*}
with proportionality coefficient~$f^\pm$ and an electrostatic interaction
potential~$\Phi$. In applications,~$f^\pm$ is sometimes also called
electrophoretic mobility and is related further to the diffusivity~$D^\pm$ by the Stokes-Einstein relation~$f^\pm=\frac{D^\pm}{kT}$,~\cite{VandeVen}. The total velocity~$v^\pm$ can therefore be
expressed by
\begin{equation*}
v^\pm = v^{\text{drift},\pm} + v^{\text{hydr}} =
-\frac{D^\pm z^\pm e}{kT}\nabla\Phi + v^{\text{hydr}}.
\end{equation*}
Inserting this expression into the standard
convection-diffusion-reaction equation for a number density~$c^\pm$
results in a modified transport equation which is also known as
Nernst-Planck equation. On the boundary~$\Gamma$ of the
considered domain~$\Omega$ we assume no-flux condition, which
supplements the so called ``no penetration'' model, described in~\cite{Elimelech}. Together with an appropriate choice of the initial
conditions~$c^{\pm,0}$, the transport of the charged particles can be
described properly by the following equations:
\begin{subequations}\label{EQU:UnScaledNernstPlanck}
\begin{align}
\partial_t c^\pm + \nabla\cdot\left(v^{\text{hydr}}c^\pm - D^\pm\nabla c^\pm - \frac{D^\pm z^\pm e}{kT}c^\pm\nabla\Phi\right) &= R^\pm\left(c\right) && \text{in } (0,T)\times\Omega \label{NernstPlanck},\\
\left(-v^{\text{hydr}}c^\pm + Dc^\pm\nabla c^\pm + \frac{D^\pm z^\pm e}{kT}c^\pm\nabla\Phi\right)\cdot\nu &= 0 && \text{on } (0,T)\times\Gamma\label{NernstPlanckRB},\\
c^\pm &= c^{\pm,0} && \text{in } \{t=0\}\times\Omega.\label{NernstPlanckAW}
\end{align}
\end{subequations}
with~$c := (c^+,c^-)$. The right-hand side~$R^\pm$
in the Nernst-Planck equation include chemical reactions between the particles, source terms et cetera.

The electrostatic interaction potential~$\Phi$ has to be calculated
using Poisson's equation~(\ref{Poisson}). The effect on the electrostatic
field implied by the charged particles themselves is included as
right-hand side. This equation may be supplemented by Neumann or Dirichlet boundary conditions which correspond to the surface charge and the so called~$\zeta$ potential of the solid matrix, respectively. Depending on the application in the geosciences either of the boundary conditions is given for example by measurements.
\begin{subequations}\label{EQU:UnScaledPoisson}
\begin{align}
-\Delta \Phi &= \frac{e}{\epsilon_0\epsilon_r}\left(z^+c^+ - z^-c^-\right) && \text{in } (0,T)\times\Omega\label{Poisson},\\
\nabla\Phi\cdot\nu &= \sigma && \text{on } (0,T)\times\Gamma_N\label{PoissonNeumannRB},\\
\Phi &= \Phi_D && \text{on } (0,T)\times\Gamma_D\label{PoissonDirichletRB}.
\end{align}
\end{subequations}
In order to determine the fluid velocity~$v^{\text{hydr}}$ we solve the modified Stokes' equations for incompressible fluid flow~(\ref{Stokes}),~\ref{StokesInkom}. As force term on the right hand side we take into account the drift force density. These equations are supplemented by a no slip boundary condition.
\begin{subequations}\label{EQU:UnScaledStokes}
\begin{align}
-\eta\Delta v^{\text{hydr}} + \frac{1}{\rho} \nabla p &= - \frac{e}{\rho}(z^+c^+ - z^-c^-)\nabla\phi^{\text{el}} && \text{in } (0,T)\times\Omega\label{Stokes}\\
\nabla\cdot v^{\text{hydr}} &=0 && \text{in } (0,T)\times\Omega\label{StokesInkom}\\
v^{\text{hydr}} &= 0 &&  \text{on } (0,T)\times\Gamma.\label{StokesRB}
\end{align}
\end{subequations}

\begin{remark}
(Part of) the system (\ref{EQU:UnScaledNernstPlanck}),
(\ref{EQU:UnScaledPoisson}), (\ref{EQU:UnScaledStokes}) arises in more general contexts. It
plays a role when determining ion distributions (for example around
colloidal particles or in a ion channel) and also in the framework
of semiconductor devices especially if the convective term is
neglected. We refer the reader to~\cite{Markovich},~\cite{Roubicek} for aspects on the modeling and analysis of the
semiconductor equations.
\end{remark}

\section{Pore Scale Model $P_\eps$}\label{SEC:PoreScaleModel}

In this section, we incorporate the physical processes described in
Section~\ref{SEC:MathMod} in a multi-scale framework and state basic properties of weak solutions as well as results concerning solvability of our problem. On the one hand, the
phenomena considered in Section~\ref{SEC:MathMod} take place on the
microscale and, on the other hand, the physical behavior we are
interested in occurs on a macroscopic domain. In the framework of
colloids, the transport takes place within the pore space of a
porous medium that is defined by its soil matrix. The definition of
the idealized underlying geometry which characterizes the highly
heterogenous porous structure is depicted in
Figure~\ref{FIG:Geometry}. The (small) scale parameter $\eps$ is
introduced to scale/balance the different terms in the governing
system of partial differential equations
(\ref{EQU:UnScaledNernstPlanck}), (\ref{EQU:UnScaledPoisson}) and (\ref{EQU:UnScaledStokes}).


\begin{figure}[h]
\centering
  \includegraphics[height=8cm]{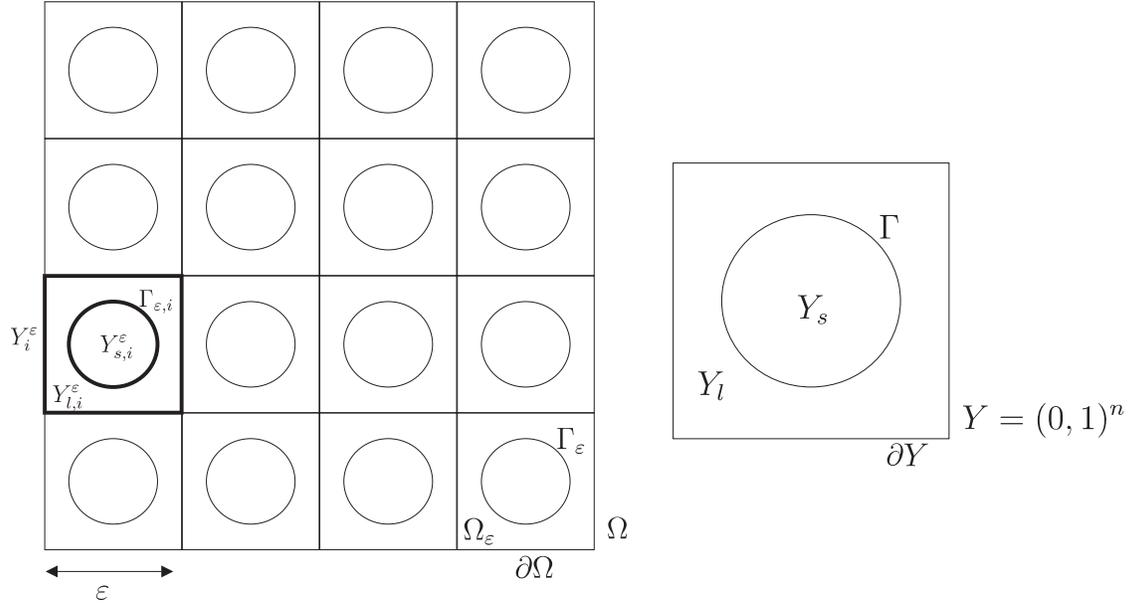}\\
  \caption{Standard unit cell (left) and periodic representation of a porous medium (right).}\label{FIG:Geometry}
\end{figure}

Let us consider a bounded and connected domain $\Omega\subset \IR^n,
\ n\in\IN$ with an associated periodic microstructure defined by the
unit cell $Y=\left(0,1\right)^n$. In the following we only consider
the physically meaningful space dimensions $n\in \{1,2,3\}$. The
unit cell $Y$ is made up of two open sets, see
Figure~\ref{FIG:Geometry}: The liquid part $Y_l$ and the solid part
$Y_s$ such that $\bar{Y}_l\cup \bar{Y}_s = \bar{Y}$ and $Y_l\cap Y_s
= \emptyset, \bar{Y}_l\cap \bar{Y}_s = \Gamma$. Especially, the solid part does not touch the boundary of the unit cell $Y$ and therefore the fluid part is connected.
We call $\eps <1$ the scale parameter and assume the macroscopic
domain to be covered by a regular mesh of size $\eps$ consisting of
$\eps$ scaled and shifted cells $Y_i^\eps$ that are divided into an
analogously scaled fluid part, solid part and boundary. Let us
denote these by $Y^\eps_{l,i}$, $Y^\eps_{s,i}$, and
$\Gamma_{\eps,i}$, respectively. The fluid part/pore space, the
solid part and the inner boundary of the porous medium are defined
by
\begin{align*}
\Omega_\eps := \bigcup_i Y^\eps_{l,i},\quad
\Omega \backslash\overline{\Omega}_\eps := \bigcup_i Y^\eps_{s,i},\quad \text{and}\quad
\Gamma_\eps := \bigcup_i \Gamma_{\eps,i}.
\end{align*}
Consequently, since we assume that $\Omega$ is completely covered by $\eps$-scaled unit cells $Y^\eps_i$ and, in particular, since the
solid part is not allowed to intersect the outer boundary, i.\,e. $\partial
\Omega\cap \Gamma_\eps = \emptyset$.


The objective of the paper is to rigorously investigate the limit $\eps\rightarrow 0$. The
focus thereby lies on the coupling between the colloidal transport, the fluid flow
and the electrostatic potential. We weight the different
terms in~(\ref{EQU:UnScaledNernstPlanck}),~(\ref{EQU:UnScaledPoisson}) and~(\ref{EQU:UnScaledStokes}) with the scale parameter~$\eps$ in order to derive reasonable macroscopic model equations. In the framework of colloids, a non-dimensionalization
procedure which can be used to motivate the choice of scaling has been done
for example in~\cite{VandeVen}. However, since the system~(\ref{EQU:UnScaledNernstPlanck}),~(\ref{EQU:UnScaledPoisson}) and~(\ref{EQU:UnScaledStokes}) is
used to describe various kinds of applications, different choices of
scaling may be interesting depending on the underlying physical
problem. We focus on the influence of the nonlinear coupling of the SNPP system due to the electrostatic potential and therefore regard Neumann as well as Dirichlet boundary condition for the Poisson equation and consider only the scaling of the coupling terms.
For the ease of presentation, we assume that~$D:=D^+=D^-$ and~$z:=z^+=-z^-$ and
suppress here the (constant) parameters~$\eta, \rho, z, e, k, T, D, \epsilon_r, \epsilon_0$ as well as the superscript~$^\text{hydr}$ within all the equations. The resulting system of
scaled partial differential equations is referred here as
Problem~$P_\eps$:
\begin{subequations}\label{EQU:ScaledSystem}
\begin{align}
- \eps^2\Delta v_\eps  + \nabla p_\eps &= - \eps^\beta (c_\eps^+ - c_\eps^-)\nabla \Phi_\eps && \text{in } (0,T)\times\Omega_\eps\label{ScaledStokes},\\
\nabla\cdot v_\eps &= 0 && \text{in } (0,T)\times\Omega_\eps\label{ScaledStokesIncomp},\\
v_\eps & = 0 && \text{on } (0,T)\times\left(\Gamma_\eps \cup \partial\Omega\right)\label{ScaledStokesRB},\\
-\eps^\alpha\Delta \Phi_\eps &= c_\eps^+ - c_\eps^- && \text{in } (0,T)\times\Omega_\eps\label{ScaledPoisson},\\
\eps^\alpha\nabla\Phi_\eps\cdot\nu &= \eps\sigma && \text{on } (0,T)\times\Gamma_{\eps,N} \label{ScaledPoissonNeumannRB},\\
\Phi_\eps &= \Phi_D && \text{on } (0,T)\times\Gamma_{\eps,D} \label{ScaledPoissonDirichletRB},\\
\eps^\alpha\nabla\Phi_\eps\cdot\nu &= 0 && \text{on } (0,T)\times\partial\Omega \label{ScaledPoissonOuterRB},\\
\partial_t c_\eps^\pm + \nabla\cdot\left(v_\eps c_\eps^\pm - \nabla c_\eps^\pm \mp \eps^\gamma c_\eps^\pm\nabla\Phi_\eps\right) &= R^\pm_\eps(c^+_\eps,c^-_\eps) && \text{in } \left(0,T\right)\times\Omega_\eps\label{ScaledNernstPlanck},\\
\left(- v_\eps c_\eps^\pm + \nabla c_\eps^\pm \pm \eps^\gamma c_\eps^\pm\nabla\Phi_\eps\right)\cdot\nu &=0 && \text{on } (0,T)\times\left(\Gamma_\eps \cup \partial\Omega\right)\label{ScaledNernstPlanckRB},\\
c_\eps^\pm &= c^{\pm,0} && \text{in } \{t=0\}\times\Omega_\eps\label{ScaledNernstPlanckAW}.
\end{align}
\end{subequations}
with the volume additivity constraint~$c^+_\eps - c^-_\eps = 1$ which is quite standard for the system~(\ref{EQU:ScaledSystem}), see e.g.~\cite{Roubicek}. This constraint can be relaxed in the case of Neumann boundary condition for the electrostatic potential, for the homogenous case see also~\cite{Schmuck}.

\begin{remark}
We could add a variable scaling also for the convective, diffusive and reactive
terms. However, we concentrate on the role of the electrostatic potential~$\Phi_\eps$.
The same
choice of scaling in the equations for~$c^\pm_\eps$ is especially
justified in the case that both types of particles have similar
properties except of the sign of the charge. On the outer boundary~$\partial\Omega$ we assume homogenous flux conditions for the concentration fields and the electrostatic potential as well as no slip boundary conditions for the velocity field. However,
different linear boundary conditions could be chosen instead without
notable changes in the calculations.
For a discussion on different
boundary conditions on the inner boundary and their influence on the
results of the homogenization procedure we refer to the discussions in Remark~\ref{REM:limitsPhiN},~\ref{REM:limitsVN},~\ref{REM:limitsCN} and~\ref{REM:limitsPhiD},~\ref{REM:limitsVD},~\ref{REM:limitsCD} and in Section~\ref{SEC:Discussion}.
\end{remark}
%
%
Multiplying the system of equations (\ref{EQU:ScaledSystem}) with the
test functions $\phi_1\in \left(H^1_0(\Omega_\eps)\right)^n,\phi_2,\phi_3,\psi\in H^1(\Omega_\eps)$ and
integrating by parts we get
the following weak formulation of Problem~$P_\eps$:
\begin{subequations}\label{EQU:weakformulation}
\begin{align}
\int_{\Omega_\eps} \eps^2\nabla v_\eps\cdot\nabla\varphi_1 -  p_\eps\nabla\cdot\varphi_1 dx = \int_{\Omega_\eps}-\eps^\beta (c^+_\eps - c^-_\eps)\nabla\Phi_\eps\cdot\varphi_1 dx\label{EQU:weakStokes}\\
\int_{\Omega_\eps} v_\eps\cdot\nabla\psi dx =0\label{EQU:weakStokesIncomp}\\
\int_{\Omega_\eps} \eps^\alpha\nabla\Phi_\eps\cdot\nabla\varphi_2 \,dx  - \int_{\Gamma_\eps} \eps^\alpha\nabla\Phi_\eps\cdot\nu\varphi_2 \,do_x = \int_{\Omega_\eps}\left(c_\eps^+ - c_\eps^-\right)\varphi_2 \,dx\label{EQU:weakPoisson},\\
\langle \partial_tc_\eps^\pm,\varphi_3\rangle_{(H^1)',H^1} +
\int_{\Omega_\eps}\!\! \left(-v_\eps c_\eps^\pm + \nabla c_\eps^\pm \pm \eps^\gamma
c_\eps^\pm\nabla\Phi_\eps\right)\cdot\nabla\varphi_3 \,dx
= \int_{\Omega_\eps} \!\! R^\pm_\eps(c^+_\eps,c^-_\eps) \varphi_3 \,dx\label{EQU:weakNernstPlanck}.
\end{align}
\end{subequations}
\begin{definition}\label{DEF:weaksolution}
We call $\left(v_\eps, p_\eps, \Phi_\eps, c_\eps^+, c_\eps^-\right)$ a weak solution of Problem
$P_ \eps$ if $v_\eps\in
L^\infty\left(0,T;H^1_0(\Omega_\eps)\right)$, $p_\eps\in
L^\infty\left(0,T;L^2(\Omega_\eps)\right)$, $\Phi_\eps\in
L^\infty\left(0,T;H^1(\Omega_\eps)\right)$ and $c^\pm_\eps \in L^\infty\left(0,T;L^2(\Omega_\eps)\right) \cap
L^2\left(0,T;H^1(\Omega_\eps)\right)$ with
$\partial_tc_\eps^\pm \in L^2\left(0,T;(H^1(\Omega_\eps))'\right)$ and
equations (\ref{EQU:weakformulation}) are satisfied for all test functions $\phi_1\in\left(H^1_0(\Omega_\eps)\right)^n, \phi_2,\phi_3,\psi\in H^1(\Omega_\eps)$.
\end{definition}


We modify the drift term in the Nernst-Planck equation by replacing the concentration fields~$c^\pm_\eps$ with the cut off functions~$\tilde{c}^\pm_\eps:= \max(0,c^\pm_\eps)$:
\begin{subequations}\label{EQU:ModifiedNernstPlanck}
\begin{align}
\partial_t c_\eps^\pm + \nabla\cdot\left(v_\eps c_\eps^\pm -  \nabla c_\eps^\pm \mp \eps^\gamma \tilde{c}_\eps^\pm\nabla\Phi_\eps\right) &= R^\pm_\eps(c^+_\eps,c^-_\eps) && \text{in } (0,T)\times\Omega_\eps\label{ModScaledNernstPlanck},\\
\left(- v_\eps c_\eps^\pm +  \nabla c_\eps^\pm \pm \eps^\gamma \tilde{c}_\eps^\pm\nabla\Phi_\eps\right)\cdot\nu &=0 && \text{in } (0,T)\times\left( \Gamma_\eps \cup \partial\Omega\right)\label{ModScaledNernstPlanckRB},\\
c_\eps^\pm &= c^{\pm,0} && \text{in } \{t=0\}\times\Omega_\eps\label{ModScaledNernstPlanckAW}.
\end{align}
\end{subequations}
The modified system consisting of~(\ref{EQU:ModifiedNernstPlanck}) and~(\ref{ScaledStokes})-(\ref{ScaledPoissonOuterRB}) is referred here as
Problem~$\tilde{P}_\eps$. The weak solution of Problem~$\tilde{P}_\eps$ is defined analogously to Definition~\ref{DEF:weaksolution}.

\begin{remark}
The weak solution of Problem~$\tilde{P}_\eps$ is also a weak solution of Problem~$P_\eps$. Furthermore, all non-negative weak solutions of Problem~$P_\eps$ are also weak solutions of Problem~$\tilde{P}_\eps$. As stated in Theorem~\ref{THM:ExistenceUniqueness} Problem~$P_\eps$ has a unique solution which is the non-negative one. Therefore both problems are equivalent.
\end{remark}


To be able to state a result on the existence and uniqueness of weak
solutions of Problem~$P_\eps$, we assume the following additional
restrictions for the ease of presentation. Especially item~2 and~4 can be relaxed. Note that, e.\,g., nonlinear
monotonic reaction terms can be handled using homogenization theory
as treated in~\cite{Hornung94}.

\begin{assumption}\label{ASS:data}\neuezeile
\begin{enumerate}
\item
On the geometry: We assume a perforated domain as introduced in Section~\ref{SEC:PoreScaleModel}, i.\,e. the pore space $\Omega_\eps$ is bounded, connected and has $C^{0,1}$-boundary.
\item On the rate coefficients: The reaction rates are assumed to have the following structure $R^\pm(c^+_\eps,c^-_\eps) = \mp(c^+_\eps - c^-_\eps)$. Especially, they are linear and employ conservation of mass for the concentration fields.
\item On the initial data: We assume the initial data to be non-negative and bounded independently of $\eps$, i.\,e.
\begin{align*}
0\leq c^{\pm,0}(x) \leq \Lambda \qquad \text{for all } x \in \Omega.
\end{align*}
Furthermore we assume the following compatibility condition for the initial data, i.e.
\begin{align*}
\int_{\Omega_\eps}c^{+,0} - c^{-,0} \,dx = \int_{\Gamma_\eps}\sigma do_x
\end{align*}
If $\sigma=0$ this implies global electro neutrality for the initial concentrations.
\item On the boundary data: We assume the boundary data $\sigma$ and $\Phi_D$ to be constant.
\end{enumerate}
\end{assumption}
In order to ensure unique weak solutions, we additionally require
\begin{assumption}\label{ASS:zeroMean}
If the electrostatic potential~$\Phi_\eps$ is determined via the equations~(\ref{ScaledPoisson}),~(\ref{ScaledPoissonNeumannRB}) and~(\ref{ScaledPoissonOuterRB}), we assume the potential~$\Phi_\eps$ to have zero mean value, i.\,e.~$\int_{\Omega_\eps}\Phi_\eps \,dx = 0$. Furthermore, we assume the pressure~$p_\eps$ to have zero mean value, i.\,e.~$\int_{\Omega_\eps}p_\eps \,dx = 0$.\\
\end{assumption}
%
%
\begin{theorem}\label{THM:MassConservation}
Let $\left(v_\eps, p_\eps,\Phi_\eps, c_\eps^+, c^-_\eps \right)$ be a weak solution of
Problem $P_\eps$ in the sense of Definition~\ref{DEF:weaksolution}. Let furthermore Assumption~\ref{ASS:data}
hold. Then the total mass $M=\int_{\Omega_\eps}c^+_\eps + c^-_\eps \,dx$ is conserved.
\end{theorem}

\begin{proof}
We test the Nernst-Planck equations~(\ref{EQU:weakNernstPlanck}) with $\varphi_3=1$, sum over $\pm$ and insert the structure of the reaction rates according to Assumption~\ref{ASS:data} which directly gives the statement of Theorem~\ref{THM:MassConservation}.
\end{proof}


\begin{theorem}\label{THM:Positivity}
Let $\left(v_\eps, p_\eps, \Phi_\eps,c_\eps^+,c_\eps^-\right)$ be a weak solution of
Problem $\tilde{P}_\eps$. Let furthermore Assumption~\ref{ASS:data}
hold. Then the concentration fields are non-negative, i.e. are bounded from below uniformly in $\eps$.
\end{theorem}

\begin{proof}
We test the Nernst-Planck equations~(\ref{EQU:weakNernstPlanck}) with $\varphi_3=(c^\pm_\eps)_- := \min (0,c^\pm_\eps)$ which yields
\begin{align*}
\int_{\Omega_\eps}\partial_tc^\pm_\eps (c^\pm_\eps)_- - v_\eps c^\pm_\eps\cdot\nabla (c^\pm_\eps)_- +  \nabla c^\pm_\eps\cdot\nabla (c^\pm_\eps)_- \pm\eps^\gamma \tilde{c}^\pm_\eps\nabla\Phi_\eps\cdot\nabla (c^\pm_\eps)_- \,dx = \int_{\Omega}R^\pm_\eps (c^\pm_\eps)_- \,dx.
\end{align*}
The drift term cancels directly due to the definition of the cut off function $\tilde{c}^\pm_\eps$. The velocity term cancels by standard calculations due to the incompressibility and no slip boundary condition. After summation over $\pm$, we have
\begin{align*}
\begin{multlined}[][0.9\linewidth]
\frac{1}{2}\frac{d}{dt} \left( \|(c^+_\eps)_-\|^2_{L^2(\Omega_\eps)} + \|(c^-_\eps)_-\|^2_{L^2(\Omega_\eps)} \right)
+   \left( \|\nabla c^+_\eps\|^2_{L^2(\Omega_\eps)} + \|\nabla c^-_\eps\|^2_{L^2(\Omega_\eps)} \right) \\
= \int_{\Omega} -(c^+_\eps - c^-_\eps) (c^+_\eps)_- + (c^+_\eps - c^-_\eps) (c^-_\eps)_- \,dx.
\end{multlined}
\end{align*}

We consider the reaction term $I_R:= -(c^+_\eps - c^-_\eps) (c^+_\eps)_- + (c^+_\eps - c^-_\eps) (c^-_\eps)_-$ for the following cases:
\begin{enumerate}
\item $c^+_\eps > 0, c^-_\eps > 0$:
$I_R = 0$
\item $c^+_\eps \leq 0, c^-_\eps > 0$:
$I_R = -(c^+_\eps - c^-_\eps)c^+_\eps \leq 0 $
\item $c^+_\eps > 0, c^-_\eps \leq 0$:
$
I_R = (c^+_\eps - c^-_\eps)c^-_\eps \leq 0
$
\item $c^+_\eps \leq 0, c^-_\eps \leq 0$:
$
I_R = -(c^+_\eps - c^-_\eps)c^+_\eps + (c^+_\eps - c^-_\eps)c^-_\eps = -(c^+_\eps - c^-_\eps)^2 \leq 0
$
\end{enumerate}

In any case we have the estimate $I_R \leq 0$ and therefore
\begin{align*}
\frac{1}{2}\frac{d}{dt} \left( \|(c^+_\eps)_-\|^2_{L^2(\Omega_\eps)} + \|(c^-_\eps)_-\|^2_{L^2(\Omega_\eps)} \right) +   \left( \|\nabla c^+_\eps\|^2_{L^2(\Omega_\eps)} + \|\nabla c^-_\eps\|^2_{L^2(\Omega_\eps)} \right) \leq 0
\end{align*}
Gronwall's lemma implies the statement of Theorem~\ref{THM:Positivity} since the initial concentrations are non-negative according to Assumption~\ref{ASS:data}.
\end{proof}


\begin{theorem}\label{THM:Boundedness}
Let $\left(v_\eps, p_\eps, \Phi_\eps, c_\eps^+, c_\eps^- \right)$ be a weak solution of
Problem $\tilde{P}_\eps$. Let furthermore Assumption~\ref{ASS:data}
hold. Then the concentration fields are bounded from above uniformly in $\eps$.
\end{theorem}

\begin{proof}
The statement of Theorem~\ref{THM:Boundedness} follows directly from Theorem~\ref{THM:Positivity} combined with the volume additivity constraint $c^+_\eps + c^-_\eps = 1$.
The boundedness of the concentration fields $c^\pm_\eps$ can be proven in the case of Neumann boundary conditions for the electrostatic potential without the volume additivity constraint, for the homogenous case see Lemma 3.3.6. in \cite{Schmuck} where Moser's iteration technique is applied formally. This formal proof can directly be extended to non-homogenous boundary conditions and linear reaction rates as defined in Assumption~\ref{ASS:data} and also be made rigorous. A rigorous approach using Moser's iteration can be found for general nonlinear equation in \cite{Lieberman}. However, an alternative and more straight forward way is to show that a maximum principle applies in the case of homogenous Neumann boundary conditions for the electrostatic potential. Since the solutions $c^\pm_\eps$ of Problem $\tilde{P}_\eps$ are non negative, $\tilde{c}^\pm_\eps$ can be replaced by $c^\pm_\eps$ in the the Nernst-Planck equations. Using $\varphi_3 = (c^\pm_\eps - \Lambda)_+ := \max (0, c^\pm_\eps - \Lambda)$ as test function, we obtain
\begin{align*}
\int_{\Omega_\eps}\partial_tc^\pm_\eps (c^\pm_\eps - \Lambda)_+ - v_\eps c^\pm_\eps\cdot\nabla (c^\pm_\eps - \Lambda)_+ +  \nabla c^\pm_\eps\cdot\nabla (c^\pm_\eps - \Lambda)_+ \pm \eps^\gamma c^\pm_\eps\nabla\Phi_\eps\cdot\nabla (c^\pm_\eps - \Lambda)_+ \,dx\\
 = \int_{\Omega_\eps}R^\pm_\eps (c^\pm_\eps - \Lambda)_+ \,dx
\end{align*}
The velocity term cancels by standard calculations due to the incompressibility and no slip boundary condition and it remains
\begin{align*}
\begin{multlined}[][0.9\linewidth]
\frac{1}{2}\frac{d}{dt}\|(c^\pm_\eps - \Lambda)_+\|^2_{L^2(\Omega_\eps)} +   \|\nabla (c^\pm_\eps - \Lambda)_+\|^2_{L^2(\Omega_\eps)} \\
\pm\eps^\gamma
\int_{\Omega_\eps} (c^\pm_\eps - \Lambda)\nabla\Phi_\eps\cdot\nabla (c^\pm_\eps - \Lambda)_+
+ \Lambda \nabla\Phi_\eps\cdot\nabla(c^\pm_\eps - \Lambda)_+ \,dx
= \int_{\Omega}R^\pm_\eps (c^\pm_\eps - \Lambda)_+ \,dx
\end{multlined}
\end{align*}
We consider the drift term separately. Using the identity $(c^\pm_\eps - \Lambda)\nabla\Phi_\eps\cdot\nabla (c^\pm_\eps - \Lambda)_+ = \nabla\Phi_\eps\cdot\frac{1}{2}\nabla (c^\pm_\eps - \Lambda)^2_+$ and integration by parts, leads to
\begin{align*}
\pm \eps^\gamma\int_{\Omega_\eps} \frac{1}{2}(-\Delta\Phi_\eps) (c^\pm_\eps - \Lambda)^2_+ + \Lambda \eps^\gamma (-\Delta\Phi_\eps)(c^\pm_\eps - \Lambda)_+ \,dx
\end{align*}
Here the homogenous Neumann boundary condition for the electrostatic potential prevents the occurrence of boundary terms.
Summation over $\pm$ and inserting the Poisson equation leads to
\begin{align*}
\int_{\Omega_\eps} \eps^\gamma \Big( (c^+_\eps - c^-_\eps) \tfrac{1}{2} (c^+_\eps - \Lambda)^2_+ - (c^+_\eps - c^-_\eps) \tfrac{1}{2}(c^-_\eps - \Lambda)^2_+
                     + \Lambda (c^+_\eps - c^-_\eps)(c^+_\eps - \Lambda)_+\vspace*{-0.5em}\\
                     - \Lambda (c^+_\eps - c^-_\eps)(c^-_\eps - \Lambda)_+\Big)\,dx
=: \int_{\Omega_\eps}T_D \,dx
\end{align*}
We now distinguish the following cases:
\begin{enumerate}
\item $c^+_\eps < \Lambda, c^-_\eps < \Lambda$: $T_D=0$
\item $c^+_\eps \geq \Lambda, c^-_\eps < \Lambda$:
$
T_D = (c^+_\eps - c^-_\eps) \frac{1}{2} (c^+_\eps - \Lambda)^2 + \Lambda (c^+_\eps - c^-_\eps)(c^+_\eps - \Lambda)  \geq 0
$
\item $c^+_\eps < \Lambda, c^-_\eps \geq \Lambda$:
$
T_D = -(c^+_\eps - c^-_\eps) \frac{1}{2} (c^-_\eps - \Lambda)^2 - \Lambda (c^+_\eps - c^-_\eps)(c^-_\eps - \Lambda)  \geq 0
$
\item $c^+_\eps \geq \Lambda, c^-_\eps \geq \Lambda$:
\begin{align*}
\begin{multlined}[]
T_D
= \frac{1}{2} (c^+_\eps - \Lambda)^3 - \frac{1}{2}(c^+_\eps - \Lambda) (c^-_\eps - \Lambda)^2 - \frac{1}{2}(c^-_\eps - \Lambda) (c^+_\eps - \Lambda)^2  + \frac{1}{2} (c^+_\eps - \Lambda)^3\\
 + \Lambda (c^+_\eps - c^-_\eps)^2 \geq 0
\end{multlined}
\end{align*}
Here we used the identity $(c^+_\eps - c^-_\eps) = (c^+_\eps - \Lambda) - (c^-_\eps - \Lambda)$ and applied Young's inequality $(3,3/2)$ which leads to a cancelation of all but the last term.
\end{enumerate}
We now consider the reaction term $T_R:= -(c^+_\eps - c^-_\eps) (c^+_\eps - \Lambda)_+ + (c^+_\eps - c^-_\eps) (c^-_\eps - \Lambda)_+$ for the following cases:
\begin{enumerate}
\item $c^+_\eps < \Lambda, c^-_\eps < \Lambda$: $T_R=0$
\item $c^+_\eps \geq \Lambda, c^-_\eps < \Lambda$:
$
T_R = -(c^+_\eps - c^-_\eps)(c^+_\eps - \Lambda) \leq 0
$
\item $c^+_\eps < \Lambda, c^-_\eps \geq \Lambda$:
$
T_R = (c^+_\eps - c^-_\eps)(c^-_\eps - \Lambda) \leq 0
$
\item $c^+_\eps \geq \Lambda, c^-_\eps \geq \Lambda$:
$
T_R = -(c^+_\eps - c^-_\eps)(c^+_\eps - \Lambda) + -(c^+_\eps - c^-_\eps)(c^-_\eps - \Lambda) = -(c^+_\eps - c^-_\eps)^2 \leq 0
$
\end{enumerate}
Finally, since $T_D\geq 0$ and $T_R \leq 0$ we have
\begin{align*}
\begin{multlined}[]
\frac{1}{2}\frac{d}{dt}\left( \|(c^+_\eps - \Lambda)_+\|^2_{L^2(\Omega_\eps)} + \|(c^-_\eps - \Lambda)_+\|^2_{L^2(\Omega_\eps)} \right)\\
+   \left( \|\nabla (c^+_\eps - \Lambda)_+\|^2_{L^2(\Omega_\eps)} + \|\nabla (c^-_\eps - \Lambda)_+\|^2_{L^2(\Omega_\eps)} \right)
\leq 0
\end{multlined}
\end{align*}
Gronwall's lemma implies the statement of Theorem~\ref{THM:Boundedness} since the initial concentrations are bounded from above by $\Lambda$ according to Assumption~\ref{ASS:data}.
\end{proof}



In the following Theorem we state \emph{a priori} estimates that are valid if we assume Neumann boundary data for the electrostatic potential on $\Gamma_\eps$. This corresponds to a physical problem in which the surface charge of the porous medium is prescribed.

\begin{theorem}\label{THM:UniformAPrioriEstimatesNeumann}
Let Assumption~\ref{ASS:data} and~\ref{ASS:zeroMean} be valid. The following \emph{a priori} estimates hold in the case of pure Neumann boundary conditions for the electrostatic potential:
\begin{align}\label{EQU:APrioriPhiNeumann}
\eps^\alpha\|\Phi_\eps\|_{L^2\left((0,T)\times\Omega_\eps\right)} +
\eps^\alpha\|\nabla\Phi_\eps\|_{L^2\left((0,T)\times\Omega_\eps\right)} &\leq C.
\end{align}
In the case $\beta-\alpha\geq 0$, it holds
\begin{align}\label{EQU:APrioriVNeumann}
\|v_\eps\|_{L^2\left((0,T)\times\Omega_\eps\right)} +
\eps\|\nabla v_\eps\|_{L^2\left((0,T)\times\Omega_\eps\right)} &\leq C.
\end{align}
If additionally $\gamma - \alpha \geq 0$ is fulfilled, it holds
\begin{align}\label{EQU:APrioriConcentrationNeumann}
\begin{multlined}[][0.9\linewidth]
\max_{0\leq t\leq T}\|c^-_\eps\|_{L^2(\Omega_\eps)} + \max_{0\leq
t\leq T}\|c^+_\eps\|_{L^2(\Omega_\eps)} + \|\nabla
c^-_\eps\|_{L^2\left((0,T)\times\Omega_\eps\right)} +  \|\nabla
c^+_\eps\|_{L^2\left((0,T)\times\Omega_\eps\right)}\\
+ \|\partial_t c^+_\eps\|_{L^2\left(0,T;(H^1(\Omega_\eps))'\right)} +
\|\partial_t c^-_\eps\|_{L^2\left(0,T;(H^1(\Omega_\eps))'\right)}
\leq C,
\end{multlined}
\end{align}
In (\ref{EQU:APrioriPhiNeumann}), (\ref{EQU:APrioriVNeumann}) and (\ref{EQU:APrioriConcentrationNeumann}), $C\in \IR_+$ is a constant independent of $\eps$.
\end{theorem}


\begin{proof}
To derive the \emph{a priori} estimates we test
(\ref{EQU:weakPoisson}) with the potential $\Phi_\eps$ which
leads to
\begin{align*}
\begin{multlined}[][0.9\linewidth]
\eps^\alpha\|\nabla\Phi_\eps\|^2_{L^2(\Omega_\eps)}
\leq \eps \|\sigma\|_{L^2(\Gamma_\eps)}\|\Phi_\eps\|_{L^2(\Gamma_\eps)} + \|c^+_\eps
- c^-_\eps\|_{L^2(\Omega_\eps)}\|\Phi_\eps\|_{L^2(\Omega_\eps)} \\
\leq \sqrt{\eps} \|\sigma\|_{L^2(\Gamma_\eps)} C\left( \|\Phi_\eps\|_{L^2(\Omega_\eps)} + \eps\|\nabla\Phi_\eps\|_{L^2(\Omega_\eps)} \right) + \|c^+_\eps
- c^-_\eps\|_{L^2(\Omega_\eps)}\|\nabla\Phi_\eps\|_{L^2(\Omega_\eps)} \\
 \leq C\left( \|\sigma\|_{L^2(\Gamma_\eps)} + \|c^+_\eps
- c^-_\eps\|_{L^2(\Omega_\eps)}\right) \|\nabla\Phi_\eps\|_{L^2(\Omega_\eps)}.
\end{multlined}
\end{align*}
Here we used $\eps\|\Phi_\eps\|^2_{L^2(\Gamma_\eps)} \leq C\left( \|\Phi_\eps\|^2_{L^2(\Omega_\eps)} + \eps^2\|\nabla\Phi_\eps\|^2_{L^2(\Omega_\eps)} \right)$ with some constant $C$ independent of $\eps$, see \cite{Hornung91} Lemma 3, Poincare's inequality for functions with zero mean value (cf.~\ref{ASS:zeroMean}) and $\eps < 1$. This results in
\begin{align*}
\eps^\alpha\|\nabla\Phi_\eps\|_{L^2(\Omega_\eps)}
& \leq C\left( \|\sigma\|_{L^2(\Gamma_\eps)} + \|c^+_\eps
- c^-_\eps\|_{L^2(\Omega_\eps)}\right)  \leq C,
\end{align*}
since $\sigma$ is constant and the concentration fields
$c^\pm_\eps$ are bounded uniformly in $\eps$, see Theorem~\ref{THM:Boundedness}. Using once again Poincar\'{e}'s
inequality leads directly to statement (\ref{EQU:APrioriPhiNeumann}) after integration with respect to time. The constant $C$ remains bounded $\eps$-independently due to Theorem~\ref{THM:Boundedness} and Assumption~\ref{ASS:data}.

We test~(\ref{EQU:weakStokes}) with the velocity field $v_\eps$ and apply
Poincar\'{e}'s inequality for functions with zero boundary values, i.\,e.~
$\|\phi_\eps\|_2\leq C_P\eps\|\nabla\phi_\eps\|_2$ with some constant
$C_P$ independent of $\eps$, see \cite{Hornung}, page 52. This
leads due to the incompressibility of $v_\eps$ and the $\eps$-independent boundedness of $c^\pm_\eps$ according to Theorem~\ref{THM:Boundedness} to
\begin{align*}
\begin{multlined}[][0.9\linewidth]
\eps^2\|\nabla v_\eps\|^2_{L^2(\Omega_\eps)}
\leq \eps^\beta  2\Lambda \|\nabla\Phi_\eps\|_{L^2(\Omega_\eps)}\|v_\eps\|_{L^2(\Omega_\eps)}
\leq \eps^\beta  C \|\nabla\Phi_\eps\|_{L^2(\Omega_\eps)}\eps\|\nabla v_\eps\|_{L^2(\Omega_\eps)}
\end{multlined}
\end{align*}
This results in
\begin{align*}
\eps\|\nabla v_\eps\|_{L^2(\Omega_\eps)}
& \leq \eps^\beta  C \|\nabla\Phi_\eps\|_{L^2(\Omega_\eps)} \leq C,
\end{align*}
if $\beta - \alpha \geq 0$, since the right hand side is bounded independently of $\eps$ due to the estimates derived for the electrostatic potential. Using once again Poincar\'{e}'s
inequality leads directly to statement (\ref{EQU:APrioriVNeumann}) after integration with respect to time and the constant $C$ remains bounded $\eps$-independently.

In Theorem~\ref{THM:Boundedness} we have already shown that $c^+_\eps$
and $c^-_\eps$ are bounded by $\Lambda$ uniformly in $\eps$. We test the Nernst-Planck equation~(\ref{EQU:weakNernstPlanck}) with $\varphi_3 = c^\pm_\eps$ to obtain an energy estimate. This allows to bound also the gradient of the concentration fields.
\begin{align*}
\begin{multlined}[][0.9\linewidth]
 \frac{1}{2}\frac{d}{dt}\|c^\pm_\eps\|^2_{L^{2}(\Omega_\eps)} +  \|\nabla c^\pm_\eps\|^2_{L^{2}(\Omega_\eps)}\\
\leq \int_{\Omega_\eps}\left| \eps^\gamma c^\pm_\eps\nabla\Phi_\eps\cdot\nabla c^\pm_\eps \right| \,dx + \int_{\Omega_\eps} R^\pm_\eps c^\pm_\eps \,dx
\leq \Lambda\eps^\gamma \|\nabla\Phi_\eps\|_{L^2(\Omega_\eps)}\|\nabla c^\pm_\eps\|_{L^2(\Omega_\eps)} + \int_{\Omega_\eps} R^\pm_\eps c^\pm_\eps \,dx\\
\leq \eps^{2\gamma - 2\alpha} C_\delta\left(\|\sigma\|^2_{L^2(\Gamma_\eps)} + \|c^+_\eps - c^-_\eps\|^2_{L^2(\Omega_\eps)}\right) + \delta \|\nabla c^\pm_\eps\|^2_{L^2(\Omega_\eps)} + \int_{\Omega_\eps} R^\pm_\eps c^\pm_\eps \,dx
\end{multlined}
\end{align*}
Here we used the estimate for the electrostatic potential derived above and that the velocity term cancels due to incompressibility of the fluid and the no slip boundary condition and Young's inequality. Summation over $\pm$, sorption with $\delta < 1/2$ and estimation of the reaction terms via $-(c^+_\eps - c^-_\eps) c^+_\eps + (c^+_\eps - c^-_\eps)c^-_\eps \leq -(c^+_\eps - c^-_\eps)^2 \leq 0$ finally leads to
\begin{align*}
\begin{multlined}[][0.9\linewidth]
\frac{1}{2}\frac{d}{dt}\left( \|c^+_\eps\|^2_{L^{2}(\Omega_\eps)} + \|c^-_\eps\|^2_{L^{2}(\Omega_\eps)} \right)
+ \frac{1}{2} \left( \|\nabla c^+_\eps\|^2_{L^{2}(\Omega_\eps)} + \|\nabla c^-_\eps\|^2_{L^{2}(\Omega_\eps)} \right)\\
\leq \eps^{2\gamma - 2\alpha} C_\delta\left(\|\sigma\|^2_{L^2(\Gamma_\eps)} + \|c^+_\eps\|^2_{L^{2}(\Omega_\eps)} + \|c^-_\eps\|^2_{L^{2}(\Omega_\eps)}\right)
\end{multlined}
\end{align*}
Integration with respect to time gives an uniform estimate of the gradient if~$\gamma - \alpha \geq 0$ since~$\sigma$ is constant and the concentration fields are bounded independently of~$\eps$.

To conclude the proof of Theorem~\ref{THM:UniformAPrioriEstimatesNeumann}, we still need to
derive estimates for the time derivatives $\partial_tc^\pm_\eps$
of the concentration fields. By the definition of the $(H^1)'$ norm and
by equations (\ref{EQU:weakNernstPlanck}), we obtain
\begin{align*}
\begin{multlined}[][0.9\linewidth]
\|\partial_t c^\pm_\eps\|_{(H^1(\Omega_\eps))'}
=  \sup_{\phi\in H^1(\Omega_\eps),\|\phi\|_{H^1(\Omega_\eps)}\leq 1} \langle \partial_t c^\pm_\eps,\phi \rangle_{(H^1)',H^1}\\
\leq \sup_{\phi\in H^1(\Omega_\eps),\|\phi\|_{H^1(\Omega_\eps)}\leq 1} \left( \left(\|c^+_\eps - c^-_\eps\|_{L^2(\Omega_\eps)} +
\Lambda\|v_\eps - \eps^\gamma
\nabla\Phi_\eps\|_{L^2(\Omega_\eps)} +  \|\nabla c^\pm_\eps\|_{L^2(\Omega_\eps)}\right) \|\phi\|_{H^1(\Omega_\eps)}\right)\\
\leq \| c^+_\eps\|_{L^2(\Omega_\eps)} + \|c^-_\eps\|_{L^2(\Omega_\eps)} +
\Lambda\|v_\eps\|_{L^2(\Omega_\eps)} + \Lambda \eps^{\gamma-\alpha}\eps^\alpha\|\nabla\Phi_\eps\|_{L^2(\Omega_\eps)} +  \|\nabla
c^\pm_\eps\|_{L^2(\Omega_\eps)}
\leq C,
\end{multlined}
\end{align*}
if $\gamma - \alpha \geq 0$ due to the uniform estimates for the gradient of the
concentration and the potential derived above, respectively. Integration with respect to time therefore yields the last statement of
Theorem~\ref{THM:UniformAPrioriEstimatesNeumann}.
\end{proof}


In the following Theorem we state \emph{a priori} estimates that are valid if we assume Dirichlet boundary data for the electrostatic potential on $\Gamma_\eps$. This corresponds to a physical problem in which the surface potential of the porous medium is prescribed. In application in the geosciences this boundary condition is related to the specification of the so called $\zeta$ potential. We define the transformed electrostatic potential $\Phi^{\text{hom}}_\eps := \Phi_\eps - \Phi_D$. Since $\Phi_D$ is a constant according to Assumption~\ref{ASS:data}, $\Phi^{\text{hom}}_\eps$ fulfills the following set of equations:
\begin{subequations}\label{EQU:TranslatedPotential}
\begin{align}
-\eps^\alpha\Delta \Phi^{\text{hom}}_\eps &= \left(c_\eps^+ - c_\eps^-\right) && \text{in } (0,T)\times\Omega_\eps\label{EQU:TranslatedPotentialEqu},
\\
\Phi^{\text{hom}}_\eps &= 0 && \text{in } (0,T)\times\Gamma_\eps,
\\
\eps^\alpha\nabla\Phi^{\text{hom}}_\eps\cdot\nu &= 0 && \text{in } (0,T)\times\partial\Omega.
\end{align}
\end{subequations}

\begin{theorem}\label{THM:UniformAPrioriEstimatesDirichlet}
Let Assumption~\ref{ASS:data} be valid. The following \emph{a priori} estimates hold in the case of Dirichlet boundary conditions on $\Gamma_\eps$ for the electrostatic potential
\begin{align}\label{EQU:APrioriTranslatedPhiDirichlet}
\eps^{\alpha-2} \|\Phi^{\text{hom}}_\eps\|_{L^2\left((0,T)\times\Omega_\eps\right)} +
\eps^{\alpha-1}\|\nabla\Phi^{\text{hom}}_\eps\|_{L^2\left((0,T)\times\Omega_\eps\right)} &\leq C.
\end{align}
In the case $\beta - \alpha + 1\geq 0$, it holds
\begin{align}\label{EQU:APrioriVDirichlet}
\|v_\eps\|_{L^2\left((0,T)\times\Omega_\eps\right)} +
\eps\|\nabla v_\eps\|_{L^2\left((0,T)\times\Omega_\eps\right)} &\leq C.
\end{align}
In the case $\gamma - \alpha + 1 \geq 0$, it holds
\begin{align}\label{EQU:APrioriConcentrationDirichlet}
\begin{multlined}[][0.9\linewidth]
\max_{0\leq t\leq T}\|c^-_\eps\|_{L^2(\Omega_\eps)} + \max_{0\leq
t\leq T}\|c^+_\eps\|_{L^2(\Omega_\eps)} + \|\nabla
c^-_\eps\|_{L^2\left((0,T)\times\Omega_\eps\right)} +  \|\nabla
c^+_\eps\|_{L^2\left((0,T)\times\Omega_\eps\right)}\\
+
\|\partial_t c^+_\eps\|_{L^2\left(0,T;(H^1(\Omega_\eps))'\right)} +
\|\partial_t c^-_\eps\|_{L^2\left(0,T;(H^1(\Omega_\eps))'\right)}
\leq C.
\end{multlined}
\end{align}
In (\ref{EQU:APrioriTranslatedPhiDirichlet}), (\ref{EQU:APrioriVDirichlet}) and (\ref{EQU:APrioriConcentrationDirichlet}), $C\in \IR_+$ is a constant independent of $\eps$.
\end{theorem}

\begin{proof}
We test equation~(\ref{EQU:TranslatedPotentialEqu}) with the translated potential $\Phi^{\text{hom}}_\eps$ and use Poincar\'{e}'s inequality for zero boundary data, see \cite{Hornung}. This leads to
\begin{align*}
\eps^\alpha\|\nabla\Phi^{\text{hom}}_\eps\|^2_{L^2(\Omega_\eps)}
\leq \|c_\eps^+ - c_\eps^-\|_{L^2(\Omega_\eps)}\|\Phi^{\text{hom}}_\eps\|_{L^2(\Omega_\eps)}
\leq \|c_\eps^+ - c_\eps^-\|_{L^2(\Omega_\eps)}\eps C_P\|\nabla\Phi^{\text{hom}}_\eps\|_{L^2(\Omega_\eps)},
\end{align*}
which results in
\begin{align*}
\eps^{\alpha-1}\|\nabla\Phi^{\text{hom}}_\eps\|_{L^2(\Omega_\eps)}
\leq C_P\|c_\eps^+ - c_\eps^-\|_{L^2(\Omega_\eps)}.
\leq C
\end{align*}
Here we have used the boundedness of the concentration fields $c^\pm_\eps$ provided by Theorem~\ref{THM:Boundedness} with $C$ being a constant
independent of $\eps$. Using again Poincar\'{e}'s inequality leads to
\begin{align*}
\eps^{\alpha-2}\|\Phi^{\text{hom}}_\eps\|_{L^2(\Omega_\eps)} \leq C
\end{align*}
Altogether, we obtain the statement~(\ref{EQU:APrioriTranslatedPhiDirichlet}) directly after integration with respect to time. By means of Theorem~\ref{THM:Boundedness}, the constant $C$ remains bounded $\eps$-independently.

The rest of the statement in Theorem~\ref{THM:UniformAPrioriEstimatesDirichlet} follows analogously to the proof of Theorem~\ref{THM:UniformAPrioriEstimatesNeumann} since due to the definition of the translated electrostatic potential and Theorem~\ref{THM:UniformAPrioriEstimatesDirichlet}, it holds $\eps^{\alpha-1}\|\nabla\Phi_\eps\|_{L^2\left(\Omega_\eps\right)} = \eps^{\alpha-1}\|\nabla\Phi^{\text{hom}}_\eps\|_{L^2\left(\Omega_\eps\right)} \leq C$.
\end{proof}


The (stationary) system consisting of (\ref{NernstPlanck}) and (\ref{Poisson}) without convective term is
well known as drift-diffusion model or van-Roosbroeck
system in the theory of semiconductor devices \cite{Roubicek}.
Analytical investigations treating existence and uniqueness of
solutions of this system can be found in \cite{Markovich} and \cite{Roubicek}.
Extensions of the system (\ref{NernstPlanck}) and (\ref{Poisson}) to the Navier-Stokes
equations have been considered analytically, for instance, in \cite{Roubicek},
\cite{Schmuck}. The results proven there can be carried over to system~(\ref{EQU:ScaledSystem}) and the following Theorem holds true:


\begin{theorem}\label{THM:ExistenceUniqueness}
Let Assumption~\ref{ASS:zeroMean} and~\ref{ASS:data} be valid. For
each $\eps > 0$ there exists a unique weak solution of Problem
$P_\eps$ in the sense of Definition~\ref{DEF:weaksolution}.
\end{theorem}

\section{Upscaling of Problem $P_\eps$}\label{SEC:TwoScaleConv}

This section is the bulk of the paper. Here we pass rigorously to
the limit $\eps\rightarrow 0$ in the non-stationary pore scale model
$P_\eps$ for both the Neumann and Dirichlet case and different choices of scaling~$(\alpha,\beta,\gamma)$. For this aim we apply the
method of two-scale convergence which has been introduced by
Nguetseng in \cite{Nguetseng} and further developed by Allaire
in \cite{Allaire}. An introduction to this topic and the
application of this method to basic model equations can be found,
for example, in \cite{CD2000} and \cite{Hornung}. For the
reader's convenience, we state the definition of two-scale
convergence as well as the basic compactness result for functions
defined on a time-space cylinder, see, e.g., \cite{Marciniak}
and \cite{Neuss-Radu}:

\begin{definition}\label{DEF:TwoScaleConv+Para}
A sequence of functions $\{\phi_\eps\}$ in $L^2\left((0,T)\times\Omega\right)$ is said to two-scale converge to a limit $\phi_0$ belonging to $L^2\left((0,T)\times\Omega\times Y\right)$ if, for any function $\psi$ in $D\left((0,T)\times\Omega;C^\infty_{\text{per}}\left(Y\right)\right)$, we have
\begin{align*}
\lim_{\eps\rightarrow 0}\int_0^T\!\!\!\int_\Omega \phi_\eps(t,x)\psi\left(t,x,\frac{x}{\eps}\right) \,dx\,dt = \int_0^T\!\!\!\int_{\Omega\times Y} \phi_0(t,x,y)\psi(t,x,y) \,dy\,dx\,dt.
\end{align*}
In short notation we write $\phi_\eps \overset{2}{\rightharpoonup} \phi_0$.\\
A sequence of functions $\{\phi_\eps\}$ in $L^2\left((0,T)\times\Gamma_\eps\right)$ is said to two-scale converge to a limit $\phi_0$ belonging to $L^2\left((0,T)\times\Omega\times \Gamma\right)$ if, for any function $\psi$ in $D\left((0,T)\times\Omega;C^\infty_{\text{per}}(\Gamma)\right)$, we have
\begin{align*}
\lim_{\eps\rightarrow 0}\eps\int_0^T\int_{\Gamma_\eps} \phi_\eps(t,x)\psi\left(t,x,\frac{x}{\eps}\right) \,do_x\,dt = \int_0^T\int_{\Omega\times \Gamma} \phi_0(t,x,y)\psi(t,x,y) \,dy\,dx\,dt.
\end{align*}
\end{definition}\noindent
Here $D\left((0,T)\times\Omega;C^\infty_{\text{per}}(Y)\right)$ and $D\left((0,T)\times\Omega;C^\infty_{\text{per}}(\Gamma)\right)$ denote the function space of infinitely smooth functions having compact support in $(0,T)\times\Omega$ with values in the space of infinitely differentiable functions that are periodic in $Y$ and $\Gamma$, respectively. The following compactness result allows to extract converging
subsequences from bounded sequences and therefore yields the
possibility to pass to the two-scale limit provided that suitable
\emph{a priori} estimates can be shown.

\begin{theorem}\label{THM:Compactness+Para}\neuezeile
\begin{enumerate}
\item
Let $\{\phi_\eps\}$ be a bounded sequence in $L^2\left((0,T)\times\Omega\right)$. Then there exists a function $\phi_0$ in $L^2\left((0,T)\times\Omega\times Y\right)$ such that, up to a subsequence, $\phi_\eps$ two-scale converges to $\phi_0$.
\item
Let $\{\phi_\eps\}$ be a bounded sequence in $L^2\left(0,T; H^1\left(\Omega\right)\right)$. Then there exist functions $\phi_0$ in $L^2\left(0,T;H^1\left(\Omega\right)\right)$ and $\phi_1$ in $L^2\left((0,T)\times\Omega; H^1_{\text{per}}\left(Y\right)\right)$ such that, up to a subsequence, $\phi_\eps$ two-scale converges to $\phi_0$ and $\nabla \phi_\eps$ two-scale converges to $\nabla_x\phi_0 + \nabla_y \phi_1$.
\item
Let $\{\phi_\eps\}$ and $\{\eps\nabla \phi_\eps\}$ be bounded sequence in $L^2\left((0,T)\times\Omega\right)$. Then there exists a function $\phi_0$ in $L^2\left((0,T)\times\Omega;H^1_{\text{per}}\left(Y\right)\right)$ such that, up to a subsequence, $\phi_\eps$ and $\eps\nabla \phi_\eps$ two-scale converge to $\phi_0$ and $\nabla_y \phi_0$, respectively.
\item
Let $\{\phi_\eps\}$ be a bounded sequence in $L^2\left((0,T)\times\Gamma_\eps\right)$. Then there exists a function $\phi_0$ in $L^2\left((0,T)\times\Omega\times\Gamma\right)$ such that, up to a subsequence, $\phi_\eps$ two-scale converges to $\phi_0$.
\end{enumerate}
\end{theorem}

\begin{proof}
For a proof of the time independent case we refer e.\,g. to
\cite{Allaire}, \cite{Neuss-Radu} and \cite{Nguetseng}. The proof can easily be
carried over to the time dependent case.
\end{proof}


One difficulty is that the \emph{a priori} estimates that have been derived in Theorem~\ref{THM:UniformAPrioriEstimatesNeumann} and Theorem~\ref{THM:UniformAPrioriEstimatesDirichlet} are at
first only valid within the perforated domain $\Omega_\eps$.
Therefore an extension of the functions $v_\eps, \nabla v_\eps, p_\eps, \Phi_\eps,
\nabla\Phi_\eps, c^\pm_\eps, \partial_t c^\pm_\eps, \nabla
c^\pm_\eps$ is necessary, such that appropriate \emph{a priori}
estimates can be extended and that the limits for $\eps\rightarrow 0$ can be identified in
function spaces on~$\Omega$. This procedure is quite standard and we refer to \cite{Allaire}, \cite{Cioranescu79}, \cite{Cioranescu99}, \cite{Hornung} and \cite{Hornung91} for the strategy and the proof of the following

\begin{theorem}\label{THM:Extensions}
For the concentration fields $c^\pm_\eps$ we apply a linear extensions operator $E\in {\cal{L}}\left(H^1(\Omega_\eps), H^1\left(\Omega\right)\right)$, such that
\begin{align*}
\|E\left(c^\pm_\eps\right)\|^2_{H^1\left(\Omega\right)} :=
\|E\left(c^\pm_\eps\right)\|^2_{L^2\left(\Omega\right)} + \|\nabla E\left(c^\pm_\eps\right)\|^2_{L^2\left(\Omega\right)}
&\leq C \|c^\pm_\eps\|^2_{H^1(\Omega_\eps)}
\end{align*}
is valid.\\
The pressure field $p_\eps$ is extended via
\begin{align*}
E(p_\eps) &:=
\begin{cases}
p_\eps & \text{in } \Omega_\eps,\\
\frac{1}{|Y^\eps_{l,i}|}\int_{Y^\eps_{l,i}}p_\eps dy & \text{in each } Y^\eps_{s,i},
\end{cases}
\end{align*}
and the following uniform \emph{a priori} estimate holds if we assume zero mean value in $\Omega$:
\begin{align*}
\|E(p_\eps)\|_{L^2((0,T)\times\Omega)} \leq C.
\end{align*}
The other variables are extended by zero into $\Omega$.
Then $\Omega_\eps$ can be replaced by $\Omega$ in the \emph{a priori} estimates from Theorem~\ref{THM:UniformAPrioriEstimatesNeumann} and Theorem~\ref{THM:UniformAPrioriEstimatesDirichlet}.
\end{theorem}

However, for the ease of presentation we suppress the notation of the extensions and write again $\varphi_\eps$ instead of $E(\varphi_\eps)$.

In the next two subsections we consider the homogenization of system
(\ref{EQU:ScaledSystem}) for both the Neumann and Dirichlet case via two-scale convergence. The
statements on the two-scale limits of the extended functions and on
the derivation of the macroscopic limit equations are deduced using
the \emph{a priori} estimates in Theorem~\ref{THM:UniformAPrioriEstimatesNeumann} and Theorem~\ref{THM:UniformAPrioriEstimatesDirichlet}.
Special attention is paid to the coupling via the electrostatic
interaction and the influence of the ranges of scaling on the limit
equations. We first state the following

\begin{definition}
We define the averaged macroscopic permittivity and diffusion tensor by
\begin{align}\label{EQU:AvTensor}
D_{ij} := \int_{Y_l}\left(\delta_{ij} +
\partial_{y_i}\phi_j\left(y\right)\right) \,dy,
\end{align}
where $\phi_j$ are solutions of the following family of cell problems ($j=1,\ldots,n$)
\begin{subequations}\label{EQU:CellProb}
\begin{align}
-\Delta_y\phi_j\left(y\right) &= 0 && \text{in } Y_l,\\
\nabla_y\phi_j\left(y\right)\cdot\nu &= -e_j\cdot\nu && \text{on } \Gamma,\\
\phi_j & && \text{periodic in } y.
\end{align}
\end{subequations}
We define the averaged macroscopic permeability tensor by
\begin{align}\label{EQU:DefAveragedPermeability}
K_{ij} = \int_{Y_l} w^i_j \,dy,
\end{align}
where $w_j$ are solutions of the following family of cell problems ($j=1,\ldots,n$)
\begin{subequations}\label{EQU:cellProblemDarcy}
\begin{align}
- \Delta_y w_j + \nabla_y\pi_j &= e_j && \text{in } Y_l\\
\nabla_y\cdot w_j &= 0 && \text{in } \Omega\times Y_l\\
w_j &= 0 && \text{in } Y_s\\
w_j & && \text{periodic in } y
\end{align}
\end{subequations}
Furthermore, we define the following cell problem
\begin{subequations}\label{EQU:CellProbDirichlet}
\begin{align}
-\Delta_y\phi\left(y\right) &= 1 && \text{in } Y_l,\\
\phi\left(y\right) &= 0 && \text{on } \Gamma,\\
\phi & && \text{periodic in } y.
\end{align}
\end{subequations}
\end{definition}

\subsection{Neumann boundary condition}\label{SEC:UpscalingNeumann}
%
We define $\tilde{\Phi}_\eps := \eps^\alpha \Phi_\eps$.
%
\subsubsection{Homogenized Limit Problems for Poisson's Equation}\label{SEC:HomPotential} 


\begin{theorem}\label{THM:limitsPhiNeumann}
Let the \emph{a priori} estimates of Theorem~\ref{THM:UniformAPrioriEstimatesNeumann} be
valid. Then the following two-scale limits can be identified for the electrostatic potential $\tilde{\Phi}_\eps$ and its gradient
$\nabla\tilde{\Phi}_\eps$: There exist functions
$\tilde{\Phi}_0 \!\in\!
L^2\!\left(0,T;H^1\!\left(\Omega\right)\right)$ and
$\tilde{\Phi}_1\!\!\in\!
L^2\!\left((0,T)\times\Omega;H^1_{\text{per}}\left(Y\right)\right)$
such that, up to a subsequence,
\begin{align*}
\tilde{\Phi}_\eps(t,x) & \overset{2}{\rightharpoonup} \tilde{\Phi}_0(t,x),\\
\nabla\tilde{\Phi}_\eps(t,x) & \overset{2}{\rightharpoonup} \nabla_x\tilde{\Phi}_0(t,x) +
\nabla_y\tilde{\Phi}_1(t,x,y).
\end{align*}
\end{theorem}

\begin{proof}
We consider the estimate (\ref{EQU:APrioriPhiNeumann}) in Theorem~\ref{THM:UniformAPrioriEstimatesNeumann} which implies
\begin{align*}
\|\tilde{\Phi}_\eps\|_{L^2\left(\Omega\right)} + \|\nabla\tilde{\Phi}_\eps\|_{L^2\left(\Omega\right)} & \leq C.
\end{align*}
Theorem~\ref{THM:Compactness+Para} ensures the existence of the two-scale limit functions.
\end{proof}


\begin{theorem}\label{THM:makro equations Phi Neumann}
Let $\left(v_\eps, p_\eps, \Phi_\eps, c_\eps^+, c_\eps^-\right)$ be a weak solution of
Problem $P_\eps$ in the sense of Definition~\ref{DEF:weaksolution}. Assume that $c^\pm_\eps$ converge strongly to $c^\pm_0$ in $L^2\left((0,T)\times\Omega\right)$.
Then the two-scale limits of $\tilde{\Phi}_\eps$ due to Theorem~\ref{THM:limitsPhiNeumann} satisfy the
following equations:
\begin{align*}
-\nabla_x\cdot\left(D\nabla_x\tilde{\Phi}_0(t,x)\right) - \bar{\sigma}_0 &= |Y_l|\left(c^+_0(t,x) - c^-_0(t,x)\right) && \text{in } (0,T)\times\Omega,\\
D\nabla_x\tilde{\Phi}_0(t,x)\cdot\nu &= 0 && \text{on } (0,T)\times\partial\Omega.
\end{align*}
\end{theorem}


\begin{proof}
To prove Theorem~\ref{THM:makro equations Phi Neumann} we test Poisson's equation~(\ref{EQU:weakPoisson}) with test function~$\left(\psi_0(t,x) \!+\! \eps\psi_1\!\left(t,x,\frac{x}{\eps}\right)\right)$ which leads to
\begin{align*}
\begin{multlined}[][0.9\linewidth]
\int_0^T\int_{\Omega} \nabla\tilde{\Phi}_\eps (t,x)\cdot\nabla\left(\psi_0(t,x) +
\eps\psi_1\left(t,x,\frac{x}{\eps}\right)\right) \,dx \,dt\\
-
\int_0^T\int_{\Gamma_\eps} \eps\sigma\left(\psi_0(t,x) +
\eps\psi_1\left(t,x,\frac{x}{\eps}\right)\right) \,dx \,dt\\
=
\int_0^T\int_{\Omega} \chi_\eps\left(x\right) \left( c_\eps^+(t,x) -
c_\eps^-(t,x) \right)\left(\psi_0(t,x) +
\eps\psi_1\left(t,x,\frac{x}{\eps}\right)\right)\,dx \,dt.
\end{multlined}
\end{align*}
We then pass to the two-scale limit $\eps\rightarrow 0$ using the properties we have
stated in Theorem~\ref{THM:limitsPhiNeumann}:
\begin{align*}
\begin{multlined}[][0.9\linewidth]
\int_0^T\int_{\Omega\times Y_l}\chi\left(y\right) \left(\nabla_x\tilde{\Phi}_0(t,x)
+ \nabla_y\tilde{\Phi}_1(t,x,y)\right)\cdot\left(\nabla_x\psi_0(t,x) +
\nabla_y\psi_1(t,x,y)\right) \,dy \,dx \,dt \\
- \int_0^T\int_{\Omega\times \Gamma}\sigma_0\psi_0(t,x) \,do_y \,dx \,dt\\
=
\int_0^T\int_{\Omega\times Y_l} \chi\left(y\right) \left(c^+_0(t,x) -
c^-_0(t,x)\right)\psi_0(t,x) \,dy \,dx \,dt
\end{multlined}
\end{align*}
Now, we choose $\psi_0(t,x) = 0$, which leads, after integration by parts
with respect to $y$, to
\begin{align*}
-\nabla_y\cdot\left(\nabla_x\tilde{\Phi}_0(t,x) + \nabla_y\tilde{\Phi}_1(t,x,y)\right) &= 0 &&\text{in } (0,T)\times\Omega\times Y_l,\\
\left(\nabla_x\tilde{\Phi}_0(t,x) + \nabla_y\tilde{\Phi}_1(t,x,y)\right)\cdot\nu &= 0 &&
\text{on } (0,T)\times\Omega\times \Gamma,\\
\tilde{\Phi}_1(t,x,y) & && \text{periodic in } y
\end{align*}
and, therefore, also to
\begin{subequations}\label{EQU:ProbPhi1}
\begin{align}
-\Delta_y\tilde{\Phi}_1(t,x,y) &= 0 &&
\text{in } (0,T)\times\Omega\times Y_l,\\
\nabla_y\tilde{\Phi}_1(t,x,y)\cdot\nu &= -\nabla_x\tilde{\Phi}_0(t,x)\cdot\nu &&
\text{on } (0,T)\times\Omega\times \Gamma,\\
\tilde{\Phi}_1(t,x,y) & && \text{periodic in } y.
\end{align}
\end{subequations}
Due to the linearity of the equation, we can deduce the following representation of $\Phi_1$:
\begin{align}\label{RepPhi1}
\tilde{\Phi}_1(t,x,y) = \sum_j\phi_j\left(y\right)\partial_{x_j}\tilde{\Phi}_0(t,x)
\end{align}
with $\phi_j$ being solutions of the standard family of
$j=1,\ldots,n$ cell problems (\ref{EQU:CellProb}).\\
On the other hand, if we choose $\psi_1(t,x,y) = 0$, we may read off, after integration by parts with
respect to $x$, the strong formulation for $\Phi_0$ :
\begin{align*}
\nabla_x\!\cdot\!\left(\int_{Y_l}\nabla_x\tilde{\Phi}_0(t,x) + \nabla_y\tilde{\Phi}_1(t,x,y)\,dy\right)
- \int_{\Gamma}\sigma_0 \,do_y
&=|Y_l|\left(c^+_0(t,x) - c^-_0(t,x)\right) && \text{in } (0,T)\!\times\!\Omega,\\
\left(\int_{Y_l}\nabla_x\tilde{\Phi}_0(t,x) + \nabla_y\tilde{\Phi}_1(t,x,y)\,dy \right)\cdot\nu & = 0 && \text{on } (0,T)\times\partial\Omega.
\end{align*}
Inserting the representation (\ref{RepPhi1}) of $\tilde{\Phi}_1$ yields
\begin{align*}
\nabla_x\cdot \left(D\nabla_x\tilde{\Phi}_0(t,x)\right) - \bar{\sigma}_0 &=
|Y_l|\left(c^+_0(t,x) - c^-_0(t,x)\right) && \text{in } \Omega,\\
D\nabla_x\tilde{\Phi}_0(t,x)\cdot\nu & = 0 && \text{on } \partial\Omega
\end{align*}
with diffusion tensor~$D$ being defined in~(\ref{EQU:AvTensor}) and~$\bar{\sigma}_0 := \int_{\Gamma}\sigma_0 \,do_y$.
\end{proof}


\begin{remark}[Modeling of $\Phi_0$]\label{REM:limitsPhiN}
In the case $\alpha = 0$, it follows $\tilde{\Phi}_\eps = \Phi_\eps$. Therefore, we have an macroscopic equation for the leading order potential $\Phi_0$ which is directly coupled to the macroscopic concentrations $c^\pm_0$. The case $\alpha < 0$ implies that $\Phi_\eps$ and $\nabla\Phi_\eps$ converge to zero. However, for any $\alpha$ an effective equation can be derived for the limit $\tilde{\Phi}_0$ of $\tilde{\Phi}_\eps$.
\end{remark}

\subsubsection{Homogenized Limit Problems for Stokes' Equation}\label{SEC:HomVelocity} 

\begin{theorem}\label{THM:limits V Neumann}
Let the \emph{a priori} estimates of Theorem~\ref{THM:UniformAPrioriEstimatesNeumann} be
valid, i.e. especially $\beta\geq\alpha$. Then the following two-scale limits can be identified for the velocity field $v_\eps$ and the gradient
$\eps\nabla v_\eps$: There exists
$v_0 \in L^2\left((0,T)\times\Omega;H^1_{\text{per}}\left(Y\right)\right)$
such that, up to a subsequence,
\begin{align*}
v_\eps(t,x) & \overset{2}{\rightharpoonup} v_0(t,x,y),\\
\eps \nabla v_\eps(t,x) & \overset{2}{\rightharpoonup} \nabla_yv_0(t,x,y).
\end{align*}
\end{theorem}

\begin{proof}
We consider the estimate~(\ref{EQU:APrioriVNeumann}) in Theorem~\ref{THM:UniformAPrioriEstimatesNeumann} which implies due to Theorem~\ref{THM:Compactness+Para} the existence of the two-scale limit functions.
\end{proof}

The convergence for~$p_\eps$ are standard, see~\cite{Hornung} and we follow directly the procedure there including the right hand side which is due to the electrostatic interaction. Depending on the choice of the scale range, this possibly leads to a coupling of the flow with the electrostatic potential and the concentration fields as stated in the following


\begin{theorem}\label{THM:makro equations V Neumann}
Let $\left(v_\eps, p_\eps, \Phi_\eps,c_\eps^+,c_\eps^-\right)$ be a weak solution of
Problem $P_\eps$ in the sense of Definition~\ref{DEF:weaksolution}. Assume that $c^\pm_\eps$ converge strongly to $c^\pm_0$ in $L^2\left((0,T)\times\Omega\right)$.\\
For $\beta\geq\alpha$ the two-scale limit of $v_\eps$ due to Theorem~\ref{THM:limits V Neumann} satisfies the
following equations:
\begin{align*}
\bar{v}_0 (t,x) &= -K\bigg(\nabla_xp_0(t,x) +
\begin{Bmatrix}
(c^+_0(t,x) - c^-_0(t,x))\nabla_x \tilde{\Phi}_0(t,x), & \beta = \alpha\\
0, & \beta > \alpha
\end{Bmatrix}
\bigg)
&& \text{in } (0,T)\times\Omega,\\
\nabla_x\cdot\bar{v}_0 (t,x) &= 0 && \text{in } (0,T)\times\Omega.
\end{align*}
\end{theorem}


\begin{proof}
Choose $\eps\psi\left(t,x,\frac{x}{\eps}\right)$ as test function:
\begin{align*}
\begin{multlined}[][0.9\linewidth]
\int_0^T\int_{\Omega} \eps\nabla v_\eps(t,x)\cdot\eps^2\nabla\psi\left(t,x,\frac{x}{\eps}\right) - p_\eps(t,x)\eps\nabla\cdot \psi\left(t,x,\frac{x}{\eps}\right) \,dx\,dt\\
= \int_0^T\int_{\Omega} -\eps^{\beta+1}(c^+_\eps(t,x) - c^-_\eps(t,x))\nabla \Phi_\eps(t,x)\psi\left(t,x,\frac{x}{\eps}\right) \,dx\,dt.
\end{multlined}
\end{align*}
Passage to the limit leads to
\begin{align*}
\int_0^T\int_{\Omega\times Y} - p_0(t,x,y)\nabla_y\cdot \psi(t,x,y) \,dy\,dx\,dt = 0,
\end{align*}
which gives $p_0(t,x,y) = p_0(t,x)$.\\
We define the space $V_\psi = \{\nabla_y\cdot\psi = 0, \nabla_x\cdot\int_{Y_l}\psi \,dy = 0, \psi = 0 \text{ on } (0,T)\times\Omega\times Y_s\}$ and choose $\psi(t,x,\frac{x}{\eps})\in V_\psi$ as test function:
\begin{align*}
\begin{multlined}[][0.9\linewidth]
\int_0^T\int_{\Omega} \eps\nabla v_\eps(t,x)\cdot\eps\nabla\psi\left(t,x,\frac{x}{\eps}\right) - p_\eps(t,x)\nabla\cdot \psi\left(t,x,\frac{x}{\eps}\right) \,dx\\
= \int_0^T\int_{\Omega} -\eps^\beta (c^+_\eps(t,x) - c^-_\eps(t,x))\nabla \Phi_\eps(t,x)\psi\left(t,x,\frac{x}{\eps}\right) \,dx
\end{multlined}
\end{align*}
Passage to the limit leads to
\begin{align*}
\int_0^T\int_{\Omega\times Y_l} \nabla_y v_0\cdot\nabla_y\psi - p_0\nabla_x\cdot \psi \,dy\,dx
=
\begin{Bmatrix}
\int_0^T\int_{\Omega\times Y_l} -(c^+_0 - c^-_0)(\nabla_x \tilde{\Phi}_0 + \nabla_y \tilde{\Phi}_1)\psi \,dy\,dx, & \beta = \alpha\\
0, & \beta > \alpha
\end{Bmatrix}
\end{align*}
Here we applied that $\psi\in V_\psi$, i.e. $\nabla_y\cdot\psi=0$ holds. The property $p_0=p_0(x)$ yields
\begin{align*}
\int_0^T\int_{\Omega} - p_0(t,x)\nabla_x\cdot \left(\int_{Y_l}\psi(t,x,y) \,dy\right)\,dx\,dt = 0
\end{align*}
Integration by parts inserting the properties of the orthogonal of $V_\psi$ and identification of the pressure $p_0$ as in \cite{Hornung} leads to
\begin{align*}
- \Delta_y v_0 + \nabla_xp_0 + \nabla_yp_1 &=
\begin{Bmatrix}
-(c^+_0 - c^-_0)(\nabla_x \tilde{\Phi}_0 + \nabla_y \tilde{\Phi}_1), & \beta=\alpha\\
0, & \beta > \alpha
\end{Bmatrix}
&& \text{in } (0,T)\times\Omega\times Y_l\\
\nabla_y\cdot v_0 &= 0 && \text{in } (0,T)\times\Omega\times Y_l\\
\nabla_x\cdot\int_{Y_l} v_0 \,dy &= 0 && \text{in } (0,T)\times\Omega\\
\int_{Y_l} v_0 \,dy \cdot\nu &= 0 && \text{on } (0,T)\times\partial\Omega\\
v_0 &= 0 && \text{on } (0,T)\times\Omega\times Y_s
\end{align*}
If $\beta = \alpha$, we define the modified pressure $\tilde{p}_1 = p_1 + (c^+_0 - c^-_0)\tilde{\Phi}_1$ in order to determine a macroscopic extended Darcy's Law. Due to the linearity of the equations $v_0$ can be represented as
\begin{align*}
v_0 (t,x,y)= - \sum_j w_j(y)\bigg(\partial_{x_j}p_0 (t,x) +
\begin{Bmatrix}
(c^+_0(t,x) - c^-_0(t,x))\partial_{x_j}\tilde{\Phi}_0(t,x), & \beta = \alpha\\
0, & \beta > \alpha
\end{Bmatrix}
\bigg)
\end{align*}
with $w_j$ being solutions of the cell problems~(\ref{EQU:cellProblemDarcy}).
We define the averaged velocity field via
\begin{align}\label{EQU:avVelocity}
\bar{v}_0(t,x) &= \int_{Y_l}v_0(t,x,y)\,dy.
\end{align}
which leads, after integration with respect to $y$, to
\begin{align*}
\bar{v}_0 (t,x) &= -K\bigg(\nabla_xp_0(t,x) +
\begin{Bmatrix}
(c^+_0 (t,x) - c^-_0(t,x))\nabla_x \tilde{\Phi}_0(t,x), & \beta = \alpha\\
0, & \beta > \alpha
\end{Bmatrix}
\bigg)
&& \text{in } (0,T)\times\Omega\\
\nabla_x\cdot\bar{v}_0 (t,x) &= 0 && \text{in } (0,T)\times\Omega
\end{align*}
with the permeability tensor~$K$ being defined in~(\ref{EQU:DefAveragedPermeability}).
\end{proof}

\begin{remark}[Modeling of $\bar{v}_0$]\label{REM:limitsVN}
In the case $\beta = \alpha$, we derive an extended incompressible Darcy's law. Besides the pressure gradient, an additional forcing term occurs due to the electrostatic potential. In the case $\beta > \alpha$, the electrostatic potential has no influence on the macroscopic velocity, which is then determined by a standard Darcy's law.
\end{remark}

\subsubsection{Homogenized Limit Problems for the Nernst-Planck Equations}\label{SEC:HomNernstPlanck} 

\begin{theorem}\label{THM:limits c Neumann}
Let the estimates of Theorem~\ref{THM:UniformAPrioriEstimatesNeumann} be valid. Then the
following two-scale limits can be identified for the concentration
fields $c^\pm_\eps$ and their gradients $\nabla c^\pm_\eps$ in the case $\gamma - \alpha \geq 0$:
There exist functions $c^\pm_0(t,x) \in L^2\left((0,T);H^1\left(\Omega\right)\right)$ and $c_1(t,x,y)\in L^2\left((0,T)\times\Omega;H^1_{\text{per}}\left(Y\right)\right)$ such that (up to a subsequence)
\begin{align*}
c^\pm_\eps(t,x) & \rightarrow c^\pm_0(t,x),\\
\nabla c^\pm_\eps(t,x) & \overset{2}{\rightharpoonup} \nabla_xc^\pm_0(t,x) +
\nabla_yc^\pm_1(t,x,y).
\end{align*}
\end{theorem}

\begin{proof}
The statement of strong convergence holds true due to the extension of the concentration fields $c^\pm_\eps$ with the properties defined in Theorem~\ref{THM:Extensions} and Aubin-Lions compact embedding lemma.
\end{proof}

\begin{remark}\label{REM:StrongConv}
The strong convergence of the concentrations $c^\pm_\eps$ in $L^2\left(0,T;L^2\left(\Omega\right)\right)$ enables us to pass to the limit
$\eps\rightarrow 0$ also in the convective and drift term of the Nernst-Planck equations (\ref{EQU:weakNernstPlanck}).
\end{remark}

\begin{theorem}\label{THM:makroEquCNeumann}
Let $\left(v_\eps, p_\eps, \Phi_\eps,c_\eps^+,c_\eps^-\right)$ be a weak solution of
Problem $P_\eps$ in the sense of Definition~\ref{DEF:weaksolution}. Assume that $\nabla\Phi_\eps$ and $v_\eps$ two-scale converge as stated in Theorem~\ref{THM:limitsPhiNeumann} and Theorem\ref{THM:limits V Neumann}, respectively.\\
Then the two-scale limits of the concentrations as stated in Theorem~\ref{THM:limits c Neumann} satisfy the
following macroscopic limit equations:
\begin{align*}
|Y_l|\partial_tc^\pm_0(t,x) + \nabla_x\cdot\bigg(\bar{v}_0(t,x)c^\pm_0(t,x) \!-\! D\nabla_x
c^\pm_0(t,x)
\pm
\begin{Bmatrix}
Dc^\pm_0(t,x) \nabla_x\tilde{\Phi}_0(t,x), & \gamma = \alpha\\
0, & \gamma > \alpha
\end{Bmatrix}
\bigg)\\
 = |Y_l|R^\pm_0(c^+_0(t,x), c^-_0(t,x))  \text{ in }\ (0,T)\!\times\!\Omega,\\
\bigg(\bar{v}_0(t,x)c^\pm_0(t,x) - D\nabla_x c^\pm_0(t,x)
\pm
\begin{Bmatrix}
Dc^\pm_0 (t,x) \nabla_x\tilde{\Phi}_0(t,x), & \gamma = \alpha\\
0, & \gamma > \alpha
\end{Bmatrix}
\bigg)
\cdot\nu  = 0  \text{ on }\
(0,T)\!\times\!\partial\Omega,
\end{align*}
\end{theorem}

\begin{proof}
We choose $\varphi_{2,3}=\psi_0(t,x) +
\eps\psi_1(t,x,\frac{x}{\eps})$ as test function in the Nernst-Planck equations~(\ref{EQU:weakNernstPlanck}) and obtain:
\begin{align*}
\begin{multlined}[][0.9\linewidth]
\int_0^T\int_{\Omega} - c^\pm_\eps(t,x) \partial_t\left(\psi_0(t,x)
+ \eps\psi_1\left(t,x,\frac{x}{\eps}\right)\right) + \left(- v_\eps(t,x) c^\pm_\eps(t,x) \right.\hspace*{-1em}\\
+ \left. \nabla c^\pm_\eps(t,x)
\pm \eps^\gamma c^\pm_\eps(t,x)\nabla\Phi_\eps(t,x) \right)
\cdot\nabla\left(\psi_0(t,x)
+ \eps\psi_1\left(t,x,\frac{x}{\eps}\right)\right) \,dx \,dt\\
= \int_0^T\int_{\Omega} R^\pm_\eps(c^+_\eps(t,x), c^-_\eps(t,x)) \left(\psi_0(t,x)
+ \eps\psi_1\left(t,x,\frac{x}{\eps}\right)\right) \,dx \,dt.
\end{multlined}
\end{align*}
Due to Theorem~\ref{THM:limits c Neumann} and Assumption~\ref{ASS:data}, we pass to the two-scale limit~$\eps\rightarrow 0$.
\begin{align*}
\begin{multlined}[][0.9\linewidth]
\int_0^T\int_{\Omega\times Y_l} - c^\pm_0(t,x)\partial_t\psi_0(t,x)
+ \bigg(- v_0(t,x,y)c^\pm_0(t,x)
+  \left(\nabla c^\pm_0(t,x) + \nabla_y c^\pm_1(t,x,y)\right)\\
\pm
\begin{Bmatrix}
c^\pm_0( \nabla_x\tilde{\Phi}_0 + \nabla_y\tilde{\Phi}_1), & \gamma = \alpha\\
0, & \gamma > \alpha
\end{Bmatrix}
\bigg)
\cdot(\nabla_x\psi_0(t,x) \!+\! \nabla_y\psi_1(t,x,y)) \,dy \,dx \,dt
\\
= \int_0^T\!\!\!\int_{\Omega\times Y_l} R^\pm_0(c^+_0(t,x), c^-_0(t,x))\psi_0(t,x) \,dy \,dx \,dt.
\end{multlined}
\end{align*}
In the case $\gamma =\alpha$ we define $\tilde{c}^\pm_1 := c^\pm_1 \pm c^\pm_0\tilde{\Phi}_1$. We choose $\psi_0 \equiv 0$, which leads, after integration by parts
with respect to $y$, to
\begin{align*}
- \Delta_y c^\pm_1(t,x,y) &= 0 && \text{in } (0,T)\times\Omega\times Y_l,\\
\nabla_y c^\pm_1(t,x,y)\cdot\nu &= -\nabla_xc^\pm_0(t,x)
\mp
\begin{Bmatrix}
c^\pm_0(t,x)\nabla_x\tilde{\Phi}_0(t,x)\cdot\nu, & \gamma = \alpha\\
0, & \gamma > \alpha
\end{Bmatrix}
&& \text{on } (0,T)\times\Omega\times \Gamma,\\
c^\pm_1(t,x,y) & && \text{periodic in } y.
\end{align*}
Due to the linearity of the equation, we deduce the following representations for $c^\pm_1$:
\begin{align}\label{RepC1}
c^\pm_1(t,x,y) = \sum_j \varphi_j(y)\partial_{x_j}c^\pm_0(t,x)
\pm
\begin{Bmatrix}
c^\pm_0\partial_{x_j}\tilde{\Phi}_0, & \gamma = \alpha\\
0, & \gamma > \alpha
\end{Bmatrix}
\end{align}
where $\varphi_j$ is the solution of the standard cell problem (\ref{EQU:CellProb}).\\
On the other hand, if we choose $\psi_1(t,x,y) = 0$, we read off the
strong formulation for $c^\pm_0$, after integration by parts with
respect to $x$, and after inserting the representation (\ref{RepC1}) of $c^\pm_1$:
\begin{align*}
|Y_l|\partial_tc^\pm_0(t,x) + \nabla_x\cdot\bigg(\bar{v}_0(t,x)c^\pm_0(t,x) \!-\! D\nabla_x
c^\pm_0(t,x)
\pm
\begin{Bmatrix}
Dc^\pm_0 \nabla_x\tilde{\Phi}_0, & \gamma = \alpha\\
0, & \gamma > \alpha
\end{Bmatrix}
\bigg)\\
 = |Y_l|R^\pm_0(c^+_0(t,x), c^-_0(t,x))  \text{ in }\ (0,T)\!\times\!\Omega,\\
\bigg(\bar{v}_0(t,x)c^\pm_0(t,x) - D\nabla_x c^\pm_0(t,x)
\pm
\begin{Bmatrix}
Dc^\pm_0 \nabla_x\tilde{\Phi}_0, & \gamma = \alpha\\
0, & \gamma > \alpha
\end{Bmatrix}
\bigg)
\cdot\nu  = 0  \text{ on }\
(0,T)\!\times\!\partial\Omega,
\end{align*}
with~$D$ and~$\bar{v}_0$ being defined in~(\ref{EQU:AvTensor}) and~(\ref{EQU:avVelocity}), respectively.
\end{proof}


\begin{remark}[Modeling of $c^\pm_0$]\label{REM:limitsCN}
Mainly two different types of limit equations arise for the macroscopic problem description. In the case $\gamma=\alpha$, the transport of the concentrations is given by Nernst-Planck equations. Thereby the limit $\tilde{\Phi}_0$ of the electrostatic potential and $\bar{v}_0$ are given in Theorem~\ref{THM:limitsPhiNeumann} and Theorem~\ref{THM:limits V Neumann}. The upscaling procedure then yields a fully coupled system of partial differential equation. In the case $\gamma>\alpha$, the electrostatic potential has no direct influence on the macroscopic concentrations. The equations for the concentrations simplify to a convection-diffusion-reaction equation. Depending on the choice of $\beta$, the effective equations might be coupled only in one direction.\\
The two families of cell problems (\ref{EQU:CellProb}) and (\ref{EQU:CellProb}) yield the same solutions and therefore the same macroscopic coefficients (up to the constant parameters that we have suppressed for the ease of presentation).
\end{remark}

\subsection{Dirichlet boundary condition}\label{SEC:UpscalingDirichlet}

\subsubsection{Homogenized Limit Problems for Poisson's Equation}\label{SEC:HomPotentialDirichlet}

We define $\tilde{\Phi}_\eps := \eps^{\alpha-2} \Phi^{\text{hom}}_\eps$ which fulfills the following set of equations:
\begin{align}\label{EQU:Psi}
-\eps^2\Delta \tilde{\Phi}_\eps &= c^+_\eps - c^-_\eps && \text{in } (0,T)\times\Omega_\eps,\\
\tilde{\Phi}_\eps &= 0 && \text{on } (0,T)\times\Gamma_{\eps} ,\\
\eps^2\nabla\tilde{\Phi}_\eps\cdot\nu &= 0 && \text{on } (0,T)\times\partial\Omega.
\end{align}


\begin{theorem}\label{THM:LimitsPhiDirichlet}
Let the \emph{a priori} estimates of Theorem~\ref{THM:UniformAPrioriEstimatesDirichlet} be
valid. Then the following two-scale limits can be identified for the
electrostatic potential $\tilde{\Phi}_\eps$ and the gradient
$\eps\nabla\tilde{\Phi}_\eps$: There exists
$\tilde{\Phi}_0\!\!\in\!
L^2\!\left((0,T)\times\Omega;H^1_{\text{per}}\left(Y\right)\right)$
such that, up to a subsequence,
\begin{align*}
\tilde{\Phi}_\eps(t,x) & \overset{2}{\rightharpoonup} \tilde{\Phi}_0(t,x,y),\\
\eps \nabla\tilde{\Phi}_\eps(t,x) & \overset{2}{\rightharpoonup} \nabla_y\tilde{\Phi}_0(t,x,y).
\end{align*}
\end{theorem}

\begin{proof}
We consider the estimate (\ref{EQU:APrioriPhiNeumann}) in Theorem~\ref{THM:UniformAPrioriEstimatesNeumann} which implies
\begin{align*}
\|\tilde{\Phi}_\eps\|_{L^2\left(\Omega\right)} + \eps\|\nabla\tilde{\Phi}_\eps\|_{L^2\left(\Omega\right)} & \leq C.
\end{align*}
Theorem~\ref{THM:Compactness+Para} then ensures the existence of the two-scale limit functions.
\end{proof}


\begin{theorem}\label{THM:makro equations Phi Dirichlet}
Let $\left(v_\eps, p_\eps, \Phi_\eps,c_\eps^+,c_\eps^-\right)$ be a weak solution of
Problem $P_\eps$ in the sense of Definition~\ref{DEF:weaksolution}. Assume that $c^\pm_\eps$ converge strongly to $c^\pm_0$ in $L^2\left((0,T)\times\Omega\right)$.
Then the two-scale limit of $\tilde{\Phi}_\eps$ due to Theorem~\ref{THM:LimitsPhiDirichlet} satisfies the
following equations:
\begin{align*}
\overline{\tilde{\Phi}}_0 (t,x) = \left(\int_{Y_l} \varphi_j(y) \,dy\right)(c^+_0(t,x) - c^-_0(t,x)).
\end{align*}
\end{theorem}


\begin{proof}
To prove Theorem~\ref{THM:makro equations Phi Dirichlet} we choose~$\psi_0\left(t,x,\frac{x}{\eps}\right)$ as test function in~(\ref{EQU:Psi}) which leads to
\begin{align*}
\begin{multlined}[][0.9\linewidth]
\int_0^T\int_{\Omega} \eps \nabla\tilde{\Phi}_\eps (t,x) \cdot\nabla\eps\psi\left(t,x,\frac{x}{\eps}\right) \,dx \,dt
=
\int_0^T\int_{\Omega} \left( c_\eps^+(t,x) -
c_\eps^-(t,x) \right) \psi\left(t,x,\frac{x}{\eps}\right) \,dx \,dt.
\end{multlined}
\end{align*}
We then pass to the two-scale limit $\eps\rightarrow 0$ using the properties we have
stated in Theorem~\ref{THM:LimitsPhiDirichlet}:
\begin{align*}
\begin{multlined}[][0.9\linewidth]
\int_0^T\int_{\Omega\times Y_l}\left(\nabla_y\tilde{\Phi}_0(t,x,y)
\cdot \nabla_y\psi(t,x,y)\right) \,dy \,dx \,dt
=
\int_0^T\int_{\Omega\times Y_l}\left(c^+_0(t,x) -
c^-_0(t,x)\right)\psi(t,x) \,dy \,dx \,dt
\end{multlined}
\end{align*}
After integration by parts with
respect to $y$, the strong formulation for $\tilde{\Phi}_0$ may be read off:
\begin{align*}
-\Delta_y\tilde{\Phi}_0 (t,x,y) &= c^+_0(t,x) - c^-_0(t,x) && \text{in } (=,T)\times\Omega\times Y_l,\\
\tilde{\Phi} &= 0 && \text{in } (0,T)\times \Omega\times \Gamma,\\
\tilde{\Phi}_0 & && \text{periodic in } y.
\end{align*}
Inserting the cell problem~(\ref{EQU:CellProbDirichlet}), we get
\begin{align*}
\overline{\tilde{\Phi}}_0 = \int_{Y_l} \tilde{\Phi}_0 \,dy = \left(\int_{Y_l} \varphi \,dy\right)(c^+_0 - c^-_0).
\end{align*}
\end{proof}


\begin{remark}[Modeling of $\Phi_0$]\label{REM:limitsPhiD}
In the case $\alpha = 2$, it follows $\tilde{\Phi}_\eps = \Phi^{\text{hom}}_0 = \Phi_\eps - \Phi_D$ and therefore
\begin{align*}
\overline{\Phi}_0 = \overline{\Phi^{\text{hom}}_0 + \Phi_D} = \int_{Y_l} \Phi^{\text{hom}}_0 + \Phi_D \,dy = \left(\int_{Y_l} \varphi \,dy\right)(c^+_0 - c^-_0) + |Y_l|\Phi_D.
\end{align*}
The macroscopic representation is directly coupled to the macroscopic concentrations $c^\pm_0$. The case $\alpha < 1$ implies that $\Phi_\eps$ and $\nabla\Phi_\eps$ converge to $\Phi_D$ and zero, respectively. However, for any $\alpha$ an effective equation can be derived for the limit $\tilde{\Phi}_0$ of $\tilde{\Phi}_\eps$.
\end{remark}

\subsubsection{Homogenized Limit Problems for Stokes' Equation}\label{SEC:HomVelocityDirichlet}

\begin{theorem}\label{THM:limitsVDirichlet}
Let the \emph{a priori} estimates of Theorem~\ref{THM:UniformAPrioriEstimatesDirichlet} be
valid, i.e. especially $\beta\geq\alpha - 1$. Then the following two-scale limits can be identified for the velocity field $\tilde{v}_\eps$ and the gradient
$\eps\nabla\tilde{v}_\eps$: There exists
$\tilde{v}_0 \in L^2\left((0,T)\times\Omega;H^1_{\text{per}}\left(Y\right)\right)$
such that, up to a subsequence,
\begin{align*}
v_\eps(t,x) & \overset{2}{\rightharpoonup} v_0(t,x,y),\\
\eps \nabla v_\eps(t,x) & \overset{2}{\rightharpoonup} \nabla_yv_0(t,x,y).
\end{align*}
\end{theorem}

The convergence for~$p_\eps$ is standard, see~\cite{Hornung} and we follow directly the procedure there including the right hand side which is due to the electrostatic interaction.


\begin{theorem}\label{THM:makro equations V Dirichlet}
Let $\left(v_\eps, p_\eps, \Phi_\eps,c_\eps^+,c_\eps^-\right)$ be a weak solution of
Problem $P_\eps$ in the sense of Definition~\ref{DEF:weaksolution}. Assume that $c^\pm_\eps$ converge strongly to $c^\pm_0$ in $L^2\left((0,T)\times\Omega\right)$.\\
For $\beta\geq\alpha-1$ the two-scale limit of $v_\eps$ due to Theorem~\ref{THM:limitsVDirichlet} satisfies the
following equations:
\begin{align*}
\bar{v}_0 (t,x) &= \int_{Y_l}v_0(t,x,y) \,dy = -K\nabla_xp_0(t,x)
&& \text{in } (0,T)\times\Omega,\\
\nabla_x\cdot\bar{v}_0 (t,x) &= 0 && \text{in } (0,T)\times\Omega.
\end{align*}
\end{theorem}
%
%
\begin{proof}
Choosing $\eps\psi(x,\frac{x}{\eps})$ as test function, it follows analogously to the proof of Theorem~\ref{THM:makro equations V Neumann} that $p_0=p_0(x)$ holds.\\
Defining the space $V_\psi = \{\nabla_y\cdot\psi = 0, \nabla_x\cdot\int\psi \,dy = 0, \psi = 0 \text{ on } (0,T)\times\Omega\times Y_s\}$
and choosing $\psi(x,\frac{x}{\eps})\in V_\psi$ as test function, leads in the limit $\eps\rightarrow 0$ to
\begin{align*}
\int_0^T\int_{\Omega\times Y_l} \nabla_y v_0\cdot\nabla_y\psi - p_0\nabla_x\cdot \psi \,dy\,dx\,dt =
\begin{Bmatrix}
\int_0^T\int_{\Omega\times Y_l} -(c^+_0 - c^-_0)\nabla_y \Phi_0\psi \,dy\,dx\,dt, & \beta = \alpha-1 \\
0,  & \beta > \alpha-1
\end{Bmatrix}
\end{align*}
We now follow the proof of Theorem~\ref{THM:makro equations V Neumann}. Finally, integration by parts results in
\begin{align*}
- \Delta_y v_0 (t,x,y) + \nabla_xp_0(t,x) + \nabla_yp_1(t,x,y) &=
\begin{Bmatrix}
-(c^+_0 (t,x)- c^-_0(t,x))\nabla_y \tilde{\Phi}_0(t,x,y),  & \beta = \alpha-1 \\
0, & \beta > \alpha-1
\end{Bmatrix}
\end{align*}
In the case $\beta = \alpha -1$, we define the modified pressure $\tilde{p}_1 = p_1 + (c^+_0 - c^-_0)\tilde{\Phi}_0$. This allows to determine a standard incompressible Darcy's Law and finishes the proof of Theorem~\ref{THM:makro equations V Dirichlet}.
\end{proof}

\begin{remark}[Modeling of $\bar{v}_0$]\label{REM:limitsVD}
The fluid flow is determined by a standard Darcy's law. The is no direct coupling to the electrostatic potential, since it is only present in the modified pressure term $\tilde{p}_1$.
\end{remark}

\subsubsection{Homogenized Limit Problems for the Nernst-Planck Equations}\label{SEC:HomNernstPlanck}

\begin{theorem}\label{THM:limits c Dirichlet}
Let the estimates of Theorem~\ref{THM:UniformAPrioriEstimatesDirichlet} be valid. Then the
following two-scale limits can be identified for the concentration
fields $c^\pm_\eps$ and their gradients $\nabla c^\pm_\eps$:
There exist functions $c_0(t,x) \in L^2\left((0,T);H^1\left(\Omega\right)\right)$ and $c_1(t,x,y)\in L^2\left((0,T)\times\Omega;H^1_{\text{per}}\left(Y\right)\right)$ such that (up to a subsequence)
\begin{align*}
c^\pm_\eps(t,x) & \rightarrow c_0(t,x),\\
\nabla c^\pm_\eps(t,x) & \overset{2}{\rightharpoonup} \nabla_xc_0(t,x) +
\nabla_yc^\pm_1(t,x,y).
\end{align*}
\end{theorem}

\begin{proof}
The statement of strong convergence holds true due to the extension of the concentration fields $c^\pm_\eps$ with the properties defined in Theorem~\ref{THM:Extensions} and Aubin-Lions compact embedding lemma.
\end{proof}

\begin{remark}\label{REM:StrongConv}
The strong convergence of the concentrations~$c^\pm_\eps$ in~$L^2\left(0,T;L^2\left(\Omega\right)\right)$ enables us to pass to the limit
$\eps\rightarrow 0$ also in the convective and drift term of the Nernst-Planck equations~(\ref{EQU:weakNernstPlanck}).
\end{remark}

\begin{theorem}\label{THM:makroEquCDirichlet}
Let $\left(v_\eps, p_\eps, \Phi_\eps,c_\eps^+,c_\eps^-\right)$ be a weak solution of
Problem $P_\eps$ in the sense of Definition~\ref{DEF:weaksolution}. Assume that $\nabla\Phi_\eps$ and $v_\eps$ two-scale converges as stated in Theorem~\ref{THM:LimitsPhiDirichlet} and Theorem~\ref{THM:limitsVDirichlet}.
Then the two-scale limits of the concentrations as stated in Theorem~\ref{THM:limits c Dirichlet} satisfy the
following macroscopic limit equations:
\begin{align*}
|Y_l|\partial_tc^\pm_0(t,x) + \nabla_x\cdot\left(\bar{v}_0(t,x)c^\pm_0(t,x) - D\nabla_x
c^\pm_0(t,x)\right)
&=
|Y_l|R^\pm_0(c^+_0(t,x), c^-_0(t,x))
&& \text{in } (0,T)\times\Omega,\\
\left(\bar{v}_0(t,x)c^\pm_0(t,x) - D\nabla_x c^\pm_0(t,x)\right)\cdot\nu &= 0 && \text{on }
(0,T)\times\partial\Omega.
\end{align*}
\end{theorem}

\begin{proof}
We choose $\phi_3= \psi_0(t,x) + \eps\psi_1 \left(t,x,\frac{x}{\eps}\right)$ as test functions in the Nernst-Planck equations~(\ref{EQU:weakNernstPlanck}).
\begin{align*}
\begin{multlined}[][0.9\linewidth]
\int_0^T\int_{\Omega}c^\pm_\eps(t,x)\partial_t\left(\psi_0(t,x) + \eps\psi_1 \left(t,x,\frac{x}{\eps}\right) \right) + \left(v_\eps(t,x)c^\pm_\eps(t,x) + \nabla c^\pm_\eps(t,x)
\pm \eps^\gamma c^\pm_\eps\nabla\Phi_\eps(t,x)\right)\\
 \cdot\nabla \left(\psi_0(t,x) + \eps\psi_1 \left(t,x,\frac{x}{\eps}\right) \right) \,dx \,dt
\\
 = \int_0^T\int_{\Omega} R^\pm_\eps(c^+_\eps(t,x), c^-_\eps(t,x))\left(\psi_0(t,x) + \eps\psi_1 \left(t,x,\frac{x}{\eps}\right)\right) \,dx \,dt.
\end{multlined}
\end{align*}
Passage to the limit $\eps\rightarrow 0$ yields
\begin{align*}
\begin{multlined}[][0.9\linewidth]
\int_0^T\!\!\!\int_{\Omega\times Y}\!\!\!\! - c^\pm_0(t,x)\partial_t\psi_0(t,x) \!+\! ( -v_0(t,x,y)c^\pm_0(t,x) \!+\! (\nabla c^\pm_0(t,x) \!+\! \nabla_y c^\pm_1(t,x,y)) \hspace*{-1em}\\
\pm
\begin{Bmatrix}
c^\pm_0\nabla_y\Phi_0(t,x,y) ) \cdot(\nabla_x\psi_0(t,x) \!+\! \nabla_y\psi_1(t,x,y)) \,dy \,dx \,dt, & \gamma=\alpha-1\\
0, & \gamma > \alpha-1
\end{Bmatrix}\\
= \int_0^T\!\!\!\int_{\Omega\times Y} R^\pm_0(c^+_0(t,x), c^-_0(t,x))\psi_0(t,x) \,dy \,dx \,dt.
\end{multlined}
\end{align*}
We define
\begin{align*}
\tilde{c}^\pm_1 = c^\pm_1
\pm
\begin{Bmatrix}
c^\pm_0\nabla_y\Phi_0, & \gamma = \alpha - 1\\
0, & \gamma > \alpha - 1
\end{Bmatrix}
\end{align*}
and choose $\psi_0 \equiv 0$, which leads, after integration by parts
with respect to $y$ to:
\begin{align*}
- \Delta_y \tilde{c}^\pm_1(t,x,y) &= 0 && \text{in } (0,T)\times\Omega\times Y_l,\\
\nabla_y \tilde{c}^\pm_1(t,x,y)\cdot\nu &= -\nabla_xc^\pm_0(t,x)\cdot\nu && \text{on } (0,T)\times\Omega\times \Gamma,\\
\tilde{c}^\pm_1(t,x,y) & &&
\text{periodic in } y.
\end{align*}
The linearity of the equation yields (\ref{RepC1}) as representations for $\tilde{c}^\pm_1$ supplemented by the family of cell problems (\ref{EQU:CellProb}).

On the other hand, if we choose $\psi_1(t,x,y) = 0$, we read off the
strong formulation for $c^\pm_0$ after integration by parts with
respect to $x$ and after inserting the representation (\ref{RepC1}) of $\tilde{c}^\pm_1$:
\begin{align*}
|Y_l|\partial_tc^\pm_0(t,x) + \nabla_x\cdot(\bar{v}_0(t,x)c^\pm_0(t,x) - D\nabla_x
c^\pm_0(t,x)
&= |Y_l|R^\pm_0(c^+_0(t,x), c^-_0(t,x)) && \text{in } (0,T)\times\Omega,\\
(-\bar{v}_0(t,x)c^\pm_0(t,x) + D\nabla_x c^\pm_0(t,x) )\cdot\nu &= 0 && \text{on }
(0,T)\times\partial\Omega,
\end{align*}
with $\bar{v}_0, D$ being defined in (\ref{EQU:AvTensor}) and (\ref{EQU:avVelocity}), respectively.
\end{proof}


\begin{remark}[Modeling of $c^\pm_0$]\label{REM:limitsCD}
The transport of the concentrations is determined by a convection-diffusion-reaction equation. The is no direct coupling to the electrostatic potential, since it is only present in the modified higher order concentration term~$\tilde{c}_1$. Depending on the choice of~$\beta$, the effective equations might be coupled only in one direction.
\end{remark}

\section{Discussion}\label{SEC:Discussion}
%
We wish to point out the following aspects:\\
In Section~\ref{SEC:TwoScaleConv}, we considered the
rigorous passage to the two-scale limit $\eps\rightarrow 0$ for
different boundary conditions of the electrostatic potential and different ranges of the scale parameter
$\left(\alpha, \beta ,\gamma\right)$ and have derived the
corresponding two-scale limits of Problem $P_\eps$. We
classified \textsl{conceptually different types of limit systems}.
In all cases, auxiliary cell problems need to be solved to be able to
provide closed-form expressions for the effective macroscopic
coefficients. Depending on chosen model, the macroscopic problem is coupled only in one direction or fully coupled. Solving these problems numerically is computationally challenging due to the mass balances that have to be fulfilled and the diverse boundary conditions, especially periodic ones. The most crucial point is that an appropriate fixed point iteration has to be constructed depending on the nature of the nonlinear couplings. Moreover, corrector estimates will be needed in order to make it possible to compare the effective solutions/problem descriptions with the oscillatory solution/microscopic model. The different structures of the resulting effective equations of the homogenization process are underlined in
Remark~\ref{REM:limitsPhiN}, Remark~\ref{REM:limitsVN}, Remark~\ref{REM:limitsCN} for Neumann boundary conditions for the electrostatic potential (i.\,e. given surface charge) and in Remark~\ref{REM:limitsPhiD}, Remark~\ref{REM:limitsVD}, Remark~\ref{REM:limitsCD} for Dirichlet boundary conditions for the electrostatic potential (i.\,e. given $\zeta$ potential). In the colloid literature, one can also find the so called perfect sink
boundary condition for the concentration fields instead of the no-penetration boundary condition, i.e. $c^\pm_\eps = 0$ on
$(0,T)\times\Gamma_\eps$. In the framework of homogenization this
would lead together with the strong convergence of the concentration
fields to $c^\pm_0\equiv 0$ as limit. Obviously, this does not provide a suitable model for colloidal transport phenomena.

The following question arises naturally: \textsl{Given a particular scenario of colloidal transport in the soil, which is the best/most reasonable mathematical (limit) model that should be considered?} Answering this question is not limited to choosing the precise
values for the choice of the appropriate boundary conditions and the scale range
$\left(\alpha, \beta,\gamma\right)$. It also requires a careful
calibration of the model by an intensive numerical testing of the
chosen set of limit equations. Further
adjustment by experimental measurements and parameter
identification procedure may need to be done to make the model
quantitatively.

It is worth noting that, using
two-scale convergence, we could not pass to the limit
$\eps\rightarrow 0$ for all choices of the parameter ranges. However, in these cases
formal two-scale asymptotic expansions can be applied in order to pass formally to
the limit $\eps\rightarrow 0$ using the
transformation $u^\pm := \exp\left(\mp\Phi\right)c^\pm$ which arises
especially when treating drift diffusion problems
(compare, e.g. \cite{Roubicek,Markovich}). An alternative is to treat a linearized system as has been considered via
rigorous homogenization in the stationary case in \cite{Allaire10}.


\section*{Acknowledgments}
N.\,R. has been funded by the Deutsche Telekom Foundation. A.\,M. has been partially supported by the Initial Training Network FIRST (Fronts and Interfaces in Science and Technology) of the European Commission under grant nr. 238702.


\end{document}